\documentclass{article}
\usepackage[utf8]{inputenc}
\usepackage[left=1in,top=1in,right=1in,bottom=1in]{geometry}
\usepackage[linktocpage=true]{hyperref}
\usepackage{setspace}
\usepackage{amssymb, amsmath, amsthm, graphicx,mathrsfs}
\usepackage{caption,cite}
\usepackage{mathtools,color}

\usepackage{cleveref}

\newtheorem{thm}{Theorem}[section]
\newtheorem{cor}[thm]{Corollary}

\newtheorem{lem}[thm]{Lemma}
\newtheorem{conj}{Conjecture}
\newtheorem{prop}[thm]{Proposition}
\newtheorem{claim}[thm]{Claim}

\newtheorem{quest}{Question}

\theoremstyle{definition}
\newtheorem{defn}{Definition}

\newcommand{\C}{\mathcal{C}}
\newcommand{\F}{\mathcal{F}}

\newcommand{\E}{\mathbb{E}}

\renewcommand{\t}{\text}

\renewcommand{\l}{\left}
\renewcommand{\r}{\right}
\newcommand{\ex}{\mathrm{ex}}

\newcommand{\Hh}{\hat{H}}

\renewcommand{\c}[1]{\mathcal{#1}}
\newcommand{\sub}{\subseteq}
\newcommand{\Om}{\Omega}

\newcommand{\f}[2]{\frac{#1}{#2}}
\newcommand{\rec}[1]{\frac{1}{#1}}
\newcommand{\al}{\alpha}
\newcommand{\ep}{\epsilon}
\newcommand{\sig}{\sigma}
\newcommand{\gam}{\gamma}
\newcommand{\Del}{\Delta}
\newcommand{\del}{\delta}
\newcommand{\half}{\frac{1}{2}}

\hypersetup{
	colorlinks,
	citecolor=blue,
	filecolor=blue,
	linkcolor=blue,
	urlcolor=blue,
	linktocpage
}

\title{Random Tur\'an Problems for Hypergraph Expansions}

\author{Jiaxi Nie\footnote{School of Mathematics, Georgia Institute of Technology {\tt jnie47@gatech.edu}.}\and Sam Spiro\footnote{Dept.\ of Mathematics, Rutgers University {\tt sas703@scarletmail.rutgers.edu}. This material is based upon work supported by the National Science Foundation Mathematical Sciences Postdoctoral Research Fellowship under Grant No. DMS-2202730.}}
\date{\today}

\begin{document}
	
	\maketitle
\begin{abstract}
    Given an $r_0$-uniform hypergraph $F$, we define its $r$-uniform expansion $F^{(r)}$ to be the hypergraph obtained from $F$ by inserting $r-r_0$ distinct vertices into each edge of $F$, and we define $\mathrm{ex}(G_{n,p}^r,F^{(r)})$ to be the largest $F^{(r)}$-free subgraph of the random hypergraph $G_{n,p}^r$.  
    
    We initiate the first systematic study of $\mathrm{ex}(G_{n,p}^r,F^{(r)})$ for general hypergraphs $F$.  Our main result essentially resolves this problem for large $r$ by showing that $\mathrm{ex}(G_{n,p}^r,F^{(r)})$ goes through three predictable phases whenever $F$ is Sidorenko and $r$ is sufficiently large, with the behavior of $\mathrm{ex}(G_{n,p}^r,F^{(r)})$ being provably more complex whenever $F$ has no Sidorenko expansion.  Moreover, our methods unify and generalize almost all previously known results for the random Tur\'an problem for degenerate hypergraphs of uniformity at least 3.    
\end{abstract}
\section{Introduction}
This paper involves the intersection of three different streams of research in extremal combinatorics: random Tur\'an problems, Sidorenko hypergraphs, and hypergraph expansions.  Throughout this paper we will work with $r$-uniform hypergraphs, or $r$-graphs for short.  Given an $r$-graph $F$, we let $v(F)$ denote the number of vertices of $F$ and $|F|$ the number of edges of $F$; and if $|F|\ge 2$, then we define its \textit{$r$-density} $d_r(F)$ by
\[d_r(F):=\max_{F'\sub F,\ |F'|\ge 2}\frac{|F'|-1}{v(F')-r}.\]

\subsection{Random Tur\'an Problems}

The \textit{Tur\'an number} $\ex(n,F)$ of an $r$-graph $F$ is defined to be the maximum number of edges that an $n$-vertex $F$-free $r$-graph can have.  Let $G_{n,p}^r$ denote the random $r$-graph on $n$ vertices obtained by including each possible edge independently and with probability $p$, and when $r=2$ we simply write $G_{n,p}$ instead of $G_{n,p}^2$.  We define the \textit{random Tur\'an number} $\ex(G_{n,p}^r,F)$ to be the maximum number of edges in an $F$-free subgraph of $G_{n,p}^r$.  Note that when $p=1$ we have $\ex(G_{n,1}^r,F)=\ex(n,F)$, so the random Tur\'an number can be viewed as a probabilistic analog of the classical Tur\'an number.

The asymptotics for $\ex(G_{n,p}^r,F)$ are essentially known if $F$ is not an $r$-partite $r$-graph due to independent breakthrough work of Conlon and Gowers \cite{conlon2016combinatorial} and of Schacht \cite{schacht2016extremal}, and because of this, we focus only on the degenerate case when $F$ is $r$-partite.  This degenerate case seems to be very difficult, even in the graph setting $r=2$ where tight bounds are known only for even cycles \cite{morris2016number,jiang2022balanced}, complete bipartite graphs~\cite{morris2016number}, and theta graphs~\cite{mckinley2023random}.  All of the tight bounds from these examples agree with the following conjecture of McKinley and Spiro~\cite{mckinley2023random}, which gives a general prediction for how the function $\ex(G_{n,p},F)$ should behave for all bipartite $F$.  To state this conjecture, we say that a sequence of events $A_n$ holds \textit{asymptotically almost surely} or \textit{a.a.s.}\ for short if $\Pr[A_n]\to 1$ as $n\to \infty$, and we write $f(n)\ll g(n)$ if $f(n)/g(n)\to 0$ as $n\to \infty$.

\begin{conj}[\cite{mckinley2023random}]\label{conj:MS}
    If $F$ is a graph with $\ex(n,F)=\Theta(n^\al)$ for some $\al\in (1,2]$, then a.a.s.\
    \[\ex(G_{n,p},F)=  \begin{cases}\max\{\Theta(p^{\al-1}n^\al), n^{2-1/d_2(F)}(\log n)^{\Theta(1)}  \} & p\gg n^{-1/d_2(F)},\\ (1+o(1))p {n\choose 2} & n^{-2}\ll p\ll n^{-1/d_2(F)}.\end{cases}\]
\end{conj}
In particular, this conjecture predicts that $\ex(G_{n,p},F)$ should always have three ranges of behavior (i.e.\ it should be roughly equal to either $p^{\al-1}n^\al,\ n^{2-1/d_2(F)}$, or $p{n\choose 2}$ depending on the value of $p$), and moreover, it predicts that one of these ranges will be a ``flat middle range,'' i.e.\  a range where $\ex(G_{n,p},F)=n^{2-1/d_2(F)}(\log n)^{\Theta(1)}$ is essentially independent of $p$ for a sizable range of $p$ close to $n^{-1/d_2(F)}$. 

Very few results are known regarding \Cref{conj:MS} for general graphs.  In particular, the only upper bounds known for $\ex(G_{n,p},F)$ for \textit{arbitrary} graphs $F$ are due to Jiang and Longbrake~\cite{jiang2022balanced}, with these general bounds matching \Cref{conj:MS} when $F$ is an even cycle.  However, it can be shown that in all cases other than $F=C_{2\ell}$, the bounds of \cite{jiang2022balanced} are  weaker than what is predicted by \Cref{conj:MS}, so more work is needed to prove the bounds of \Cref{conj:MS} in every other case.

Turning now from graphs to hypergraphs, we note that not much is known about the random Tur\'an problem for $r$-graphs with $r\ge 3$.  Indeed, the only such $F$ for which effective (though typically non-tight) bounds are known for are complete $r$-partite $r$-graphs~\cite{spiro2022counting}, loose even cycles~\cite{mubayi2023random,nie2024turan}, Berge cycles~\cite{nie2024turan,spiro2022counting}, expansions of hypergraph cliques and subgraphs of tight trees \cite{nie2021triangle,nie2023random}, and non-Sidorenko hypergraphs \cite{nie2023sidorenko}.

\subsection{Sidorenko Hypergraphs}
Part of the difficulty in proving results about random Tur\'an numbers for hypergraphs  is that no analog of \Cref{conj:MS} is known to hold for $r$-graphs,   making it difficult to predict a priori what the expected behavior of $\ex(G_{n,p}^r,F)$ should be for a given $r$-graph $F$. 
 Moreover, this lack of a general conjecture for the random Tur\'an problem of hypergraphs turns out to be closely related to the lack of an analog of Sidorenko's conjecture for hypergraphs.  To discuss this further, we need to formally define what it means for a hypergraph to be Sidorenko.

\begin{defn}
    A \textit{homomorphism} from an $r$-graph $F$ to an $r$-graph $H$ is a map $\phi:V(F)\to V(H)$ such that $\phi(e)$ is an edge of $H$ whenever $e$ is an edge of $F$.  We let $\hom(F,H)$ denote the number of homomorphisms from $F$ to $H$ and define the \textit{homomorphism density}
\[t_F(H)=\frac{\hom(F,H)}{v(H)^{v(F)}}.\]
    We say that an $r$-graph $F$ is \textit{Sidorenko} if for every $r$-graph $H$ we have
    \[t_F(H)\ge t_{K_r^r}(H)^{|F|},\]
    where $K_r^r$ is the $r$-graph consisting of a single edge.
\end{defn}

The study of Sidorenko hypergraphs is an extremely active area of research within extremal combinatorics, especially in the context of graphs~\cite{conlon2010approximate,conlon2017finite,conlon2018sidorenko,conlon2018some,coregliano2021biregularity,fox2017local,hatami2010graph,kim2016two,li2011logarithimic,lovasz2011subgraph,szegedy2014information}, with the most outstanding problem in this direction being the following.

\begin{conj}[Sidorenko's conjecture \cite{Sidorenko1991Inequalities,Sidorenko1993Acorrelation}]
    Every bipartite graph is Sidorenko.
\end{conj}

It is well known that Sidorenko's conjecture does not extend to hypergraphs, in the sense that there exist $r$-partite $r$-graphs $F$ which are not Sidorenko for all $r\ge 3$.  For such hypergraphs, it was shown by Conlon, Lee, and Sidorenko~\cite{conlon2023extremal} that $\ex(n,F)$ is always strictly larger than the trivial bound given by a simple deletion argument.  This result was later extend to the random Tur\'an setting in \cite{nie2023sidorenko}, where in particular it was shown that essentially any counterexample to Sidorenko's conjecture would provide a counterexample to \Cref{conj:MS}, and more generally that any non-Sidorenko $r$-partite $r$-graph would provide a counterexample to the hypergraph analog of \Cref{conj:MS}.

\subsection{Hypergraph Expansions}
The discussion above illustrates that getting a complete understanding of random Tur\'an numbers of arbitrary hypergraphs is complicated by the existence of non-Sidorenko hypergraphs.  To partially overcome this, we will restrict ourselves in this paper to the following class of hypergraphs which is known to behave well with respect to both random Tur\'an numbers and being Sidorenko.

\begin{defn}
    Given an $r_0$-graph $F$, we define its \textit{$r$-expansion} $F^{(r)}$ to be the $r$-graph obtained by inserting $r-r_0$ distinct new vertices into each edge of $F$.  When $F=C_\ell$ is a graph cycle we will sometimes refer to $F^{(r)}=C_{\ell}^{(r)}$ as a \textit{loose cycle}.
\end{defn}
Hypergraph expansions are a natural and well-studied class of hypergraphs, especially in the context of (classical) Tur\'an numbers.  For example, the famed Erd\H{o}s matching conjecture~\cite{erdos1965problem} is equivalent to determining $\ex(n,M_k^{(r)})$ when $M_k$ is a graph matching of size $k$.  There are many more results for Tur\'an numbers of expansions \cite{chung1983unavoidable,kostochka2015turan,mubayi2007intersection,pikhurko2013exact}, and we refer the interested reader to the survey by Mubayi and Verstra\"ete~\cite{mubayi2016survey} for much more on (classical) Tur\'an problems for expansions.

One nice features about hypergraph expansions is that their Tur\'an numbers are easy to compute for large $r$.  In particular, it is straightforward to show for any $r_0$-graph $F$ which does not contain a vertex contained in every edge of $F$ that $\ex(n,F^{(r)})=\Theta_r(n^{r-1})$ for all $r$ sufficiently large in terms of $F$.  In view of this, it is natural to ask if one can also easily establish bounds for random Tur\'an numbers of expansions, at least in the case when $r$ is large and $p$ is relatively close to 1.  And indeed, there have been a number of works studying special cases of this problem \cite{mubayi2023random,nie2021triangle,nie2023random,nie2024turan}, with perhaps the most well studied case being that of loose even cycles $C_{2\ell}^{(r)}$. 

The particular problem of studying loose even cycles was initially studied by Mubayi and Yepremyan~\cite{mubayi2023random}, with later independent work also being done by Nie~\cite{nie2024turan}. 
 In total, these two streams of research culminated in the following set of tight bounds for this problem whenever the uniformity is at least $4$.

\begin{thm}[\cite{nie2024turan,mubayi2023random}]\label{thm:looseCycles}
    For every $r\ge 4$ and $\ell\ge2$, we have a.a.s.
    $$
    \ex(G^r_{n,p}, C^{(r)}_{2\ell})=
    \l\{
    \begin{aligned}
        & \max\{\Theta(pn^{r-1}),n^{1+\frac{1}{2\ell-1}}(\log n)^{\Theta(1)}\},~~&\t{if}~p\gg n^{-r+1+\frac{1}{2\ell-1}}\\
        &(1+o(1))p\binom{n}{r},~~&\t{if}~ n^{-r}\ll p\ll  n^{-r+1+\frac{1}{2\ell-1}}.\\
    \end{aligned}
    \r.
    $$
\end{thm}

We emphasize here the peculiar situation that tight bounds for loose cycles are unknown when $r=3$, despite tight bounds being known for $r=2$ due to \cite{morris2016number} and for $r\ge 4$ due to \Cref{thm:looseCycles}.   We also note that very recently, Jiang and Longbrake~\cite{jiang2024number} improved upon the bounds of \Cref{thm:looseCycles} by obtaining tight asymptotic bounds for $r\ge 5$ and $p$ sufficiently large through developing a supersaturation version of the celebrated delta-system method.

For general hypergraph expansions, the behavior of $\ex(G_{n,p}^r,F^{(r)})$ when both $r$ and $p$ are large was determined by Nie~\cite{nie2023random}.

\begin{thm}[Theorem 1.3,~\cite{nie2023random}]\label{thm:weakExpansion}
If $F$ is an $r_0$-graph, then for all $r\ge v(F)$ and\\ $p\ge n^{-\frac{r-v(F)}{r-v(F)+1}\cdot \frac{1}{d_r(F^{(r)})}}(\log n)^{\Omega(1)}$, we have a.a.s.
$$
\ex(G^r_{n,p}, F^{(r)})=\Theta\l(pn^{r-1}\r).
$$
\end{thm}
In fact, a somewhat stronger result is proven in \cite{nie2023random} where the $v(F)$ term throughout \Cref{thm:weakExpansion} can be replaced by a somewhat smaller integer $k$, but in either case the theorem fails to determine the behavior of $\ex(G_{n,p}^r,F^{(r)})$ for all ranges of $p$ for any hypergraph $F$ outside of certain classes of hypergraph trees.


This failure to determine the random Tur\'an number for all values of $p$ and nearly all hypergraphs is largely due to the fact that $\ex(G_{n,p}^r,F^{(r)})$ can exhibit wildly differing behaviors for different choices of $F$.  For example, if $F$ is an even cycle, then \Cref{thm:looseCycles} shows that $\ex(G_{n,p}^r,F^{(r)})=\Theta(pn^{r-1})$ for all $p\gg n^{\frac{-1}{d_r(F^{(r)})}}$.  On the other hand, if $F=C_3$ then it is known \cite{nie2021triangle,nie2023random} that $\ex(G_{n,p},F^{(r)})\gg p n^{r-1}$ for $n^{\frac{-1}{d_r(F^{(r)})}}\ll p \ll n^{-\frac{r-v(F)}{r-v(F)+1}\cdot \frac{1}{d_r(F^{(r)})}}$, which in particular shows that the range of $p$ for which \Cref{thm:weakExpansion} holds can not be increased for arbitrary $F$.  More generally, it is known that $\ex(G_{n,p}^r,F^{(r)})$ will be substantially larger than $pn^{r-1}$ for certain ranges of $p\gg n^{\frac{-1}{d_r(F^{(r)})}}$ whenever $F^{(r)}$ is not Sidorenko, and in these cases the behavior of $\ex(G_{n,p}^r,F^{(r)})$ is largely unknown even in the simple case of $F=C_5$.

\section{Main Results}

Our main result is a significant extension of \Cref{thm:weakExpansion}, essentially classifying which hypergraphs $F$ are such that the random Tur\'an numbers $\ex(G_{n,p}^r,F^{(r)})$ for large $r$ exhibits three simple ranges of behavior analogous to the ranges predicted in \Cref{conj:MS} for graphs.  To state this classification, we say that an $r_0$-graph $F$ with $|F|\ge 3$ is \textit{strictly $r_0$-balanced} if $d_{r_0}(F)=\frac{|F|-1}{v(F)-r_0}$ and if $\frac{|F'|-1}{v(F')-r_0}<d_{r_0}(F)$ for every proper subgraph $F'\sub F$ with at least two edges.  We also let $\Del(F)$ denote the maximum degree of $F$.  
\begin{thm}\label{thm:strongExpansion}
    Let $F$ be a strictly $r_0$-balanced $r_0$-graph with $2\le \Del(F)<|F|$.  If $F$ is Sidorenko, then for all $r> |F|^2v(F)r_0$, we have a.a.s.
    $$
    \ex(G^r_{n,p}, F^{(r)})=
    \l\{
    \begin{aligned}
        & \max\l\{\Theta(pn^{r-1}),n^{r_0-\frac{1}{d_{r_0}(F)}}(\log n)^{\Theta(1)}\r\},~~&\t{if}~p\gg n^{-r+r_0-\frac{1}{d_{r_0}(F)}}\\
        &(1+o(1))p\binom{n}{r},~~&\t{if}~ n^{-r}\ll p\ll  n^{-r+r_0-\frac{1}{d_{r_0}(F)}}.\\
    \end{aligned}
    \r. $$
    Moreover, if there does not exist any $k$ such that $F^{(k)}$ is Sidorenko, then these bounds fail to hold for all $r$.
\end{thm}

As an aside, it is unknown whether or not there exist hypergraphs $F$ which are not Sidorenko and which have some expansion $F^{(k)}$ that is Sidorenko~\cite[Question 5.2]{nie2023sidorenko}, so the hypothesis in the``moreover'' part of this statement may in fact be equivalent to saying that $F$ is not Sidorenko.

While the bound $r>|F|^2v(F)r_0$ from \Cref{thm:strongExpansion} is not very large given the high level of generality of the theorem statement, it is natural to ask if these bounds can be improved for specific choices of $F$.  In particular, the following problem is of interest.
\begin{quest}\label{prob:Sidorenko}
    Which Sidorenko $r_0$-graphs $F$ are such that $\ex(n,F)=\omega(n^{r_0-1})$ and are such that the tight bounds of \Cref{thm:strongExpansion} hold for every $r$ with $\ex(n,F^{(r)})=\Theta(n^{r-1})$?
\end{quest}

The hypothesis that $\ex(n,F^{(r)})=\Theta(n^{r-1})$ is necessary for the bounds of \Cref{thm:strongExpansion} to hold, as otherwise the stated bound for $p=1$ would be incorrect.  The condition $\ex(n,F)=\omega(n^{r_0-1})$ is introduced primarily to avoid trivial solutions to this problem, since without this condition, every large expansion of a Sidorenko hypergraph would give a positive answer by \Cref{thm:strongExpansion}.

As far as we are aware, there are no known examples for which either positive or negative answers to \Cref{prob:Sidorenko} are known, and in particular, it may very well be the case that \Cref{prob:Sidorenko} has a positive answer for \textit{every} Sidorenko hypergraph.  The closest the current literature comes to verifying any case of \Cref{prob:Sidorenko} is \Cref{thm:looseCycles} for $F$ an even cycle, giving tight bounds for all $r\ge 4$ (c.f.\ \Cref{prob:Sidorenko} which asks for $r\ge 3$ in this case).  

Although we remain unable to answer \Cref{prob:Sidorenko} even for the simplest even cycle $C_4=K_{2,2}$, we can give a positive answer for every other $K_{2,t}$, giving the first (infinite number of) positive answers to \Cref{prob:Sidorenko}.

\begin{thm}\label{theorem:K2t_RT}
For all integers $r,t\ge 3$ and all $0<p=p(n)\le 1$ we have a.a.s.
$$
    \ex(G^r_{n,p}, K_{2,t}^{(r)})=
    \l\{
    \begin{aligned}
        & \max\{\Theta(pn^{r-1}),n^{\frac{3t-1}{2t-1}}(\log n)^{\Theta(1)}\},~~&\t{if}~p\gg n^{-r+\frac{3t-1}{2t-1}}\\
        &(1+o(1))p\binom{n}{r},~~&\t{if}~ n^{-r}\ll p\ll  n^{-r+\frac{3t-1}{2t-1}}.\\
    \end{aligned}
    \r.
    $$
\end{thm}

More generally, we can significantly improve upon the dependency on $r$ from \Cref{thm:strongExpansion} for all complete bipartite graphs $K_{s,t}$, though optimal bounds only hold for $s=2$.

\begin{thm}\label{thm:Kst}
    If $r,s,t\ge 2$ are integers such that $t\ge s$ and either $t\ge s^2-2s+3$ and $r\ge s+1$ or $r\ge s+2$, then for all $0<p=p(n)\le 1$ we have a.a.s.
    $$
    \ex(G^r_{n,p}, K_{s,t}^{(r)})=
    \l\{
    \begin{aligned}
        & \max\{\Theta(pn^{r-1}),n^{2-\frac{s+t-2}{st-1}}(\log n)^{\Theta(1)}\},~~&\t{if}~p\gg n^{-r+2-\frac{s+t-2}{st-1}}\\
        &(1+o(1))p\binom{n}{r},~~&\t{if}~ n^{-r}\ll p\ll  n^{-r+2-\frac{s+t-2}{st-1}}.\\
    \end{aligned}
    \r.
    $$
\end{thm}
It is possible that these tight bounds for $\ex(G_{n,p}^r,K_{s,t}^{(r)})$ continue to hold for any $r$ such that ${r\choose 2}\ge s$, as this is precisely the range where $\ex(n,K_{s,t}^{(r)})=\Theta(n^{r-1})$ for $t$ sufficiently large \cite[Theorem 1.4]{ma2018some}, but \Cref{thm:Kst} is the best bound for $r$ that can be obtained with our general methods and the known (tight) random Tur\'an bounds for the graph $K_{s,t}$ \cite{morris2016number}.  This being said, in forthcoming work \cite{nie202X}, we use a significantly more technical argument tailored specifically towards $K_{s,t}$ expansions in order to go beyond \Cref{thm:Kst} and prove the same result for $r$ at least roughly $2s/3+1$, in particular giving another positive answer to \Cref{prob:Sidorenko} for $K_{4,t}$ and $t$ sufficiently large.

Finally, we give a positive solution to \Cref{prob:Sidorenko} for theta graphs with sufficiently many paths.  Here the \textit{theta graph } $\theta_{a,b}$ is defined to be the graph consisting of $a$ internally disjoint paths of length $b$ between a common pair of vertices $u,v$.  For example, $\theta_{2,b}=C_{2b}$ and $\theta_{a,2}=K_{2,a}$.  Our general methods give good bounds on $\theta_{a,b}^{(r)}$ whenever effective bounds for the random Tur\'an number of $\theta_{a,b}$ are known and $a\ge 3$ (see \Cref{prop:thetaBS}), and in particular, recent results of McKinley and Spiro~\cite{mckinley2023random} give the following.
\begin{thm}\label{theorem:theta_RT}
If $r\ge 3$, $a\ge 100$, and $b\ge 3$ are integers, then for all $0<p=p(n)\le 1$, we have a.a.s.

 $$
    \ex(G^r_{n,p}, \theta_{a,b}^{(r)})=
    \l\{
    \begin{aligned}
        & \max\{\Theta(pn^{r-1}),n^{1+\frac{a-1}{ab-1}}(\log n)^{\Theta(1)}\},~~&\t{if}~p\gg n^{-r+1+\frac{a-1}{ab-1}}\\
        &(1+o(1))p\binom{n}{r},~~&\t{if}~ n^{-r}\ll p\ll  n^{-r+1+\frac{a-1}{ab-1}}.\\
    \end{aligned}
    \r.
    $$
\end{thm}

\subsection{General Tools for Random Tur\'an Problems}
In the process of proving the results detailed above, we will establish a number of general tools for random Tur\'an problems (listed out below) which we expect will be of use in future study into this area, especially for problems related to hypergraph expansions.  In particular, we note that this collection of results are enough to fairly easily recover every previously known upper bound for random Tur\'an problems of expansions \cite{mubayi2023random,nie2021triangle,nie2023random,nie2024turan}. 

We begin with a very basic result that esatablishes effective bounds for arbitrary hypergraphs $F$.

\begin{prop}\label{lem:generalLowerRandomTuran}
    Let $F$ be an $r$-graph with $\Delta(F)\ge 2$.  If $n^{-r}\ll p \ll n^{-\frac{1}{d_r(F)}}$, then a.a.s.
    \[\ex(G_{n,p}^r,F)=(1+o(1))p {n\choose r}.\]
    If $p\gg  n^{-\frac{1}{d_r(F)}}$, then a.a.s.
    \[\ex(G_{n,p}^r,F)\ge  \max\{\Omega(p\cdot \ex(n,F)),n^{r-\frac{1}{d_r(F)}}(\log n)^{-1}\}.\]
\end{prop}

With \Cref{lem:generalLowerRandomTuran} in hand, all that is needed to prove tight bounds for $\ex(G_{n,p}^r,F)$ are effective upper bounds on the problem whenever $p\gg n^{-1/d_r(F)}$, which is usually the most challenging component of random Tur\'an problems.  Establishing such bounds typically requires proving a ``balanced supersaturation'' result, which informally says that if an $r$-graph $H$ has many edges, then one can find a collection $\c{H}$ of copies of $F$ in $H$ which is ``spread out'' in the sense that no set of edges in $H$ lies in too many copies of $F$ in $\c{H}$.  

More precisely, we will make use of the following definition, where here for a hypergraph $\c{H}$ and $S\sub V(\c{H})$ we define the degree $\deg_{\c{H}}(S)$ to be the number of edges of $\c{H}$ containing $S$, and we define $\Del_i(\c{H})=\max\{\deg_{\c{H}}(S):S\sub V(\c{H}),\ |S|=i\}$. 

\begin{defn}
    Given positive functions $M=M(n),\ \gamma=\gamma(n)$, and $\tau=\tau(n,m)$, we say that an $r$-graph $F$ is \emph{$(M,\gamma,\tau)$-balanced} if for every $n$-vertex $m$-edge $r$-graph $H$ with $n$ sufficiently large and $m\ge M(n)$, there exists a non-empty collection $\c{H}$ of copies of $F$ in $H$ (which we view as an $|F|$-uniform hypergraph on $E(H)$) such that, for all integers $1\le i\le |F|$,
    \begin{equation}\label{equation:Delta}
    \Delta_i(\c{H})\le \frac{\gamma(n)|\c{H}|}{m}\l(\frac{\tau(n,m)}{m}\r)^{i-1}.        
    \end{equation}
\end{defn}
In other words, $F$ is $(M,\gamma,\tau)$-balanced if for every $r$-graph with sufficiently many vertices and a large number of edges (i.e.\ $m\ge M(n)$), we can find a collection of copies of $F$ which are spread out in the sense of \eqref{equation:Delta}. 

This definition together with a standard application of hypergraph containers yields the following result. 

\begin{prop}\label{Lemma:General Random Turan}
Let $F$ be an $r$-graph with $r\ge2$.  If there exists a  $C>0$ and positive functions $M=M(n)$ and $\tau=\tau(n,m)$ such that 
\begin{itemize}
    \item[(a)] $F$ is $(M,(\log n)^{C}, \tau)$-balanced, and
    \item[(b)] For all sufficiently large $n$ and $m\ge M(n)$, the function $\tau(n,m)$ is non-increasing with respect to $m$ and satisfies $\tau(n,m)\ge 1$, 
\end{itemize}
then there exists $C'\ge 0$ such that for all sufficiently large $n$, $m\ge M(n)$, and $0<p\le 1$ with $pm\rightarrow\infty$ as $n\rightarrow \infty$, we have a.a.s.
$$
\ex(G^r_{n,p},F)\le \max\l\{C'pm,\tau(n,m)(\log n)^{C'}\r\}.
$$
\end{prop}
We emphasize that condition (b) here is quite mild: for any $\tau$ such that (a) holds, we can always apply \Cref{Lemma:General Random Turan} after replacing $\tau$ with the new function $\tau'(n,m):=1+\max_{m'\ge m}\tau(n,m')$ since (a) continues to hold if we increase $\tau$.

With \Cref{Lemma:General Random Turan} in hand, we see that to prove our main results we need only show that Sidorenko expansions $F^{(r)}$ are balanced in some appropriate sense.  To do this, we establish several ``Lifting Theorems'' which show that expansions $F^{(r)}$ are balanced whenever the base hypergraph $F$ is.  These set of Lifting Theorems provide significant generalizations to similar results of Mubayi and Yepremyan~\cite{mubayi2023random} and Nie~\cite{nie2024turan} and can be viewed as the main technical innovation of the present work.

We begin with the first result in this direction.  While the statement is rather technical, we emphasize that it is simply saying that if $F$ is $(M,\gamma,\tau)$-balanced for some reasonable choice of $M,\gamma,\tau$, then every expansion $F^{(r)}$ is also balanced in some appropriate way.

\begin{thm}\label{lemma:BSviaSHADOW}
Let $r>r_0\ge 2$ and let $F$ be an $r_0$-graph such that there exist positive functions $M_{r_0}=M_{r_0}(n)$, $\gamma_{r_0}=\gamma_{r_0}(n)$ and $\tau_{r_0}=\tau_{r_0}(n,m)$ with the following properties:
\begin{itemize}
    \item [(a)] $F$ is $(M_{r_0}, \gamma_{r_0}, \tau_{r_0})$-balanced,
    \item[(b)] For sufficiently large $n$, the function $M_{r_0}(n)n^{-\frac{r-r_0}{r-1}}$ is non-decreasing with respect to $n$,
    \item[(c)] For sufficiently large $n$, the function  $\gamma_{r_0}(n)$ is non-decreasing with respect to $n$.
    \item [(d)] For sufficiently large $n$,  $m\ge M_{r_0}(n)$, and any real $x\ge 1$, the function $\tau_{r_0}(n,m)$ is non-increasing with respect to $m$ and $\tau_{r_0}(nx,mx^{\frac{r-r_0}{r-1}})$ is non-decreasing with respect to $x$.
\end{itemize}
Then there exists a sufficiently large constant $C$ depending only on $F$ and $r$ such that for the functions
\begin{equation*}
    M_r(n):=\max\l\{M_{r_0}(n)^{\frac{r-1}{r_0-1}}n^{-\frac{r-r_0}{r_0-1}},~n^{r-1}\r\}(\log n)^{C},
\end{equation*}

\begin{equation*}
\gamma_r(n):=\gamma_{r_0}(n)(\log n)^{C},
\end{equation*}
and
\begin{equation*}
\tau_r(n,m):=\tau_{r_0}\l(n,n^{\frac{r-r_0}{r-1}}m^{\frac{r_0-1}{r-1}}(\log n)^{-C}\r)(\log n)^{C},
\end{equation*}
we have that $F^{(r)}$ is $(M_r,\gamma_r,\tau_r)$-balanced.
\end{thm}

Our second Lifting Theorem gives stronger bounds than \Cref{lemma:BSviaSHADOW} when going from $F^{(r-1)}$ to $F^{(r)}$ but can only be applied when $r$ is sufficiently large.  To state this, we require an additional definition.

\begin{defn}\label{def:tree}
    An $r$-graph $T$ is said to be a {\em tight $r$-tree} if its edges can be ordered as $h_1,\dots,~h_t$ so that 
$$
\forall i\ge 2~\exists v\in h_i~and~1\le s\le i-1~such~that~v\not\in\cup_{j=1}^{i-1}h_j~and~h_i-v\subset h_s.
$$ 
\end{defn}

\begin{thm}\label{lemma:BSviaGreedy}
Let $F$ be an $(r-1)$-graph with $r-1,|F|\ge 2$ such that $F^{(r)}$ is a spanning subgraph of a tight $r$-tree and such that there exist positive functions $M_{r-1}=M_{r-1}(n)$, $\gamma_{r-1}=\gamma_{r-1}(n)$ and $\tau_{r-1}=\tau_{r-1}(n,m)$ with the following properties: 
\begin{itemize}
    \item[(a)] $F$ is $(M_{r-1},\gamma_{r-1},\tau_{r-1})$-balanced,
    \item[(b)] For sufficiently large $n$ the function $\gamma_{r-1}(n)$ is non-decreasing with respect to $n$,
    \item[(c)] For sufficiently large $n$ and $m\ge M_{r-1}(n)$, the function  $\tau_{r-1}(n,m)$ is non-decreasing with respect to $n$ and non-increasing with respect to $m$.
    
\end{itemize}
Then there exists a sufficiently large constant $C$ depending only on $F$  such that the following holds. Let 
\[M_r(n):=Cn^{r-1},\]
\begin{equation*}
\gamma_{r}(n):=\gamma_{r-1}(n)\cdot C\log n,  
\end{equation*}
and let $A=A(n,m)$ be any function such that, for all sufficiently large $n$ and $m\ge M_r(n)$,
\begin{itemize}
    \item[(d)]
    \begin{equation*}
        m n^{1-r}\le A(n,m)\le \frac{m}{M_{r-1}(n)\cdot C\log n},
    \end{equation*}
\end{itemize}
 
and let
\begin{equation*}
\tau_{r}(n,m):=\max\l\{A(n,m)^{-\frac{1}{d_r\l(F^{(r)}\r)}}\cdot m,~\tau_{r-1}\l(n,~\frac{m}{A(n,m)\cdot C\log n}\r)\cdot C\log n\r\}.   
\end{equation*}
Then $F^{(r)}$ is $(M_r,\gamma_{r},\tau_{r})$-balanced.
\end{thm}

\Cref{lemma:BSviaGreedy} turns out to be essentially at least as strong as \Cref{lemma:BSviaSHADOW} whenever it applies; see the arXiv-only appendix of this paper for more on this. 

By using \Cref{lemma:BSviaGreedy}, we will be able to show that if $F$ is an $r_0$-graph which has ``optimal'' balanced supersaturation (see \Cref{prop:optimalBSBound}) and Tur\'an number $O(n^{r_0-1})$ (i.e.\ one whose extremal construction has the same size as that of an $r_0$-uniform star), then every expansion of $F$ has essentially optimal bounds for the random Tur\'an problem in the following sense.  


\begin{thm}\label{thm:optimalBalanced}
Let $F$ be an $r_0$-graph with $2\le \Delta(F)<|F|$ which is the spanning subgraph of a tight $r_0$-tree. 
 If there exists a constant $C_{r_0}$ such that $F$ is $(M_{r_0},\gamma_{r_0},\tau_{r_0})$-balanced where $M_{r_0}(n)=C_{r_0}n^{r_0-1}$, $\gamma_{r_0}(n)= (\log n)^{C_{r_0}}$ and $\tau_{r_0}(n,m)= n^{r_0-\frac{1}{d_{r_0}(F)}}(\log n)^{C_{r_0}}$; then for all $r\ge r_0$, we have
     $$
    \ex(G^r_{n,p}, F^{(r)})=
    \l\{
    \begin{aligned}
        & \max\{\Theta(pn^{r-1}),n^{r-\frac{1}{d_r(F^{(r)})}}(\log n)^{\Theta(1)}\},~~&\t{if}~p\gg n^{-\frac{1}{d_r(F^{(r)})}}\\
        &(1+o(1))p\binom{n}{r},~~&\t{if}~ n^{-r}\ll p\ll  n^{-\frac{1}{d_r(F^{(r)})}}.\\
    \end{aligned}
    \r.
    $$
\end{thm}

This result gives a partial explanation to the peculiar situation mentioned around \Cref{thm:looseCycles}
about  tight random Tur\'an bounds for $C_{2\ell}^{(r)}$ being known only for $r=2$ and $r\ge 4$:\ \Cref{thm:optimalBalanced} shows (roughly) that larger uniformity expansions of $F$ are easier to prove bounds for once the Tur\'an number of these expansions become sufficiently small.

\textbf{Organization}.  The rest of our paper is organized as follows.  We prove Theorems~\ref{lemma:BSviaSHADOW} and \ref{lemma:BSviaGreedy} in \Cref{sec:BS}, which is the bulk of the technical work of this paper.  We then prove several results related to optimal balanced supersaturation in \Cref{sec:optimal} and general random Tur\'an results in \Cref{sec:Turan},  culminating in a proof of \Cref{thm:optimalBalanced}. Finally, we apply our Lifting Theorems  to prove our main results on $\ex(G_{n,p}^r,F^{(r)})$ for various Sidorenko $F$ in \Cref{sec:applying}.
\section{Balanced Supersaturation for Expansions}\label{sec:BS}
In this section we prove our two general results Theorems~\ref{lemma:BSviaSHADOW} and \ref{lemma:BSviaGreedy} about lifting balanced supersaturation from $F$ to the expansions $F^{(r)}$.  For this, given an $r$-graph $H$ we say that a set of vertices $S$ with $|S|=r_0$ is an \textit{$r_0$-shadow} of $H$ if there exists an edge of $H$ containing the set $S$. 
\subsection{Proof Sketch}\label{subsec:sketch}

To illustrate our methods, we informally consider the toy problem of finding copies of $C_4^{(3)}$ in a 3-graph $H$ where every pair of vertices in $H$ has degree either 0 or $d$.  In this setting we use one of two approaches to build our copies of $C_4^{(3)}$.

The first approach grows a $C_4^{(3)}$ from a copy of $C_4$ in the shadow graph  $\partial H$, say a $C_4$ which has vertex set $v_1,v_2,v_3,v_4$.  Because $\{v_1,v_2\}$ has degree $d$ in $H$, there exist at least $d-4$ choices for a vertex $u_{1,2}\notin \{v_1,v_2,v_3,v_4\}$ such that $\{v_1,v_2,u_{1,2}\}$ is an edge in $H$.  Similarly there are at least $d-8$ choices for each of the vertices $u_{2,3},u_{3,4},u_{1,4}$ such that $\{v_i,v_j,u_{i,j}\}$ are edges in $H$ and such that all of these vertices are distinct, and these vertices in total define a copy of $C_4^{(3)}$ in $H$.  With this procedure, each graph $C_4$ in $\partial H$ can be used to build roughly $d^4$ copies of $C_4^{(3)}$ in $H$.  As such, if we can find a collection $\c{H}'$ of copies of $C_4$ in $\partial H$, then we can use this to find a collection $\c{H}$ of copies of $C_4^{(3)}$ in $H$ of size roughly $d^4 |\c{H}'|$, and moreover, it will turn out that $\Del_i(\c{H})$ will be  small provided $\Del_i(\c{H}')$ is, from which we will conclude that balanced supersaturation for $C_4$ leads to balanced supersaturation for $C_4^{(3)}$.

The second approach involves growing a $C_4^{(3)}$ from an arbitrary edge $h=\{v_1,v_2,v_3\}$ in $H$.  Because $v_1,v_3$ has degree $d$, we can choose in $d-3$ ways some $v_4\notin \{v_1,v_2,v_3\}$ which is in an edge with $v_1,v_3$, after which we can perform the exact same expansion procedure as above to build a total of roughly $d^5$ copies of $C_4^{(3)}$ from this edge $h$.  In total this approach gives a collection of $C_4^{(3)}$'s of size about $d^5 |H|$, which will do better than the previous approach depending on the relative sizes of $d$ and $|H|$.  We emphasize here that while our first approach works for essentially any $F$ and $r$, this second approach will only work if the uniformity of $H$ is sufficiently large for the tight tree condition to hold.

We informally refer to the first approach described here as ``Shadow Expansion'' and the second as ``Greedy Expansion''.  We will prove \Cref{lemma:BSviaSHADOW} by using only shadow expansion, and we will prove \Cref{lemma:BSviaGreedy} by combining both of these methods. 

\subsection{Shadow Expansion: Proof of \Cref{lemma:BSviaSHADOW}}
We begin with our main lemma for the Shadow Expansion procedure described above.  Recall that given an $r$-graph $H$, we say that a set of vertices $S$ with $|S|=r_0$ is an \textit{$r_0$-shadow} of $H$ if there exists an edge of $H$ containing the set $S$.

\begin{lem}\label{lemma:building from the shadow}
Let $r>r_0\ge 2$.  Let $F$ be an $r_0$-graph such that  there exist functions $M=M(n)$, $\gamma=\gamma(n)$ and $\tau=\tau(n,m)$ with the property that
\begin{itemize}
    \item[(A)] $F$ is $(M,\gamma,\tau)$-balanced.
\end{itemize}
Let $H$ be an $n$-vertex $r$-partite $r$-graph with $r$-partition $V(H)=V_1\cup\dots\cup V_r$. Let $\hat{n}=|V_1\cup\dots\cup V_{r_0}|$ and let $\hat{m}$ be the number of $r_0$-shadows of $H$ in $V_1\cup\dots\cup V_{r_0}$. If
\begin{itemize}
    \item [(B)] $\hat{n}$ is sufficiently large,
    \item[(C)] $\hat{m}\ge M(\hat{n})$, and if
    \item[(D)] there exist real numbers $d$ and $D$ with $D\ge d\ge 2v(F^{(r)})n^{r-r_0-1}$ such that for every $r_0$-shadow $S$ of $H$ in $V_1\cup\dots\cup V_{r_0}$, 
    $$d\le \deg_{H}(S)\le D,$$
\end{itemize}
then there exists a collection $\c{H}$ of copies of $F^{(r)}$ in $H$ such that for all $ 1\le i\le |F|$,
$$
\Delta_i(\c{H})\le \frac{\gamma(\hat{n})|\c{H}|}{D \hat{m}}\cdot\l(\frac{\tau(\hat{n},\hat{m})}{D\hat{m}}\r)^{i-1}\cdot \l(\frac{2^{r+1}D}{d}\r)^{|F|}.
$$

\end{lem}
\begin{proof}
Let $H_{[r_0]}$ be the $r_0$-graph on $V_1\cup\dots\cup V_{r_0}$ whose edges are all of the $r_0$-shadows of $H$ in $V_1\cup\dots\cup V_{r_0}$. Note that $v(H_{[r_0]})=\hat{n}$ and $|H_{[r_0]}|=\hat{m}$.  By properties (a), (b) and (c), there exists a collection $\hat{\c{H}}$ of $F$ in $H_{[r_0]}$ such that for all $ 1\le i\le |F|$,
\begin{equation}\label{equation:shadow_InductiveAssumption}
\Delta_i(\hat{\c{H}})\le \frac{\gamma(\hat{n})|\hat{\c{H}}|}{\hat{m}}\cdot\l(\frac{\tau(\hat{n},\hat{m})}{\hat{m}}\r)^{i-1}.    
\end{equation}

We now obtain a collection $\c{H}$ of copies of $F^{(r)}$ in $H$ by ``expanding'' each copy of $F$ in $\hat{\c{H}}$.  More precisely, consider the following procedure for constructing copies of $F^{(r)}$ in $H$. Start by picking some $\hat{F}\in \hat{\c{H}}$ and give an arbitrary ordering $e_1,\ldots,e_m$ to the edges of $\hat{F}$.  We then iteratively for each $1\le k\le m$ pick $h_k\in H$ to be any edge of $H$ which contains $e_k$ and which is disjoint from $(V(\hat{F})\cup \bigcup_{k'<k} h_{k'})\setminus \{e_k\}$, after which the process terminates.  Observe that if this process terminates, then these $h_k$ edges form a copy of $F^{(r)}$ in $H$, and we let $\c{H}$ denote the set of copies of $F^{(r)}$ that can be formed in this way.  It remains to analyze the properties of this collection.

Observe that in the procedure, the number of choices for $h_k$ given the edge $e_k$ is crudely at least \begin{equation}\deg_H(e_k)-v(F^{(r)})n^{r-r_0-1},\label{eq:weakShadowDegreeBound}\end{equation} since $|(V(\hat{F})\cup \bigcup_{k'<k} h_{k'})\setminus \{e_k\}|\le v(F^{(r)})$ and since there are at most $n^{r-r_0-1}$ edges of $H$ containing $e_k$ and any given vertex from $(V(\hat{F})\cup \bigcup_{k'<k} h_{k'})\setminus \{e_k\}$.  Since \eqref{eq:weakShadowDegreeBound} is at least $d/2$ by property (d), we find that the total number of ways of going through the procedure is at least $|\hat{\c{H}}|(d/2)^{|F|}$.  Moreover, since each copy of $F^{(r)}$ can be produced in at most $2^{r|F|}$ ways from this procedure (since there are crudely at most ${r\choose r_0}^{|F|}$ ways of picking an $\hat{F}$ which produces $F^{(r_0)}$), we in total find that

\begin{equation}\label{equation:shadow_LowerboundOfH}
|\c{H}|\ge |\hat{\c{H}}|\l(\frac{d}{2^{r+1}}\r)^{|F|}.
\end{equation}

Now fix some set $\sig\sub H$ of $i$ edges of $H$ and observe that one can specify each copy of $F^{(r)}$ in $\c{H}$ containing $\sig$ through the following two steps: 
\begin{itemize}
    \item[1.] Specify a copy $\hat{F}\in \hat{\c{H}}$ which contains all $i$ of the $[r_0]$-shadows of $\sig$ in $\bigcup_{i\in[r_0]}V_i$. There are at most $\Delta_i(\hat{\c{H}})$ ways to do this.
    \item[2.] For each of the $|F|-i$ other $[r_0]$-shadows of $\hat{F}$ that are not in $\sig$, we specify an edge in $H$ which contains it. By property (d), for each $[r_0]$-shadow there are at most $D$ ways to do this.
\end{itemize}
In total then, we see for all $1\le i\le |F|$ that

\begin{equation}\label{equation:shadow_UpperboundOfDeltaH}
\Delta_i(\c{H})\le \Delta_{i}(\hat{\c{H}})D^{|F|-i}.
\end{equation}

Finally, by (\ref{equation:shadow_UpperboundOfDeltaH}), (\ref{equation:shadow_InductiveAssumption}), and (\ref{equation:shadow_LowerboundOfH}), we have, for all $1\le i\le |F|$,
$$
\begin{aligned}
    \Delta_i(\c{H})&\le \frac{\gamma(\hat{n})|\hat{\c{H}}|}{\hat{m}}\cdot\l(\frac{\tau(\hat{n},\hat{m})}{\hat{m}}\r)^{i-1}\cdot D^{|F|-i}\\
    &\le \frac{\gamma(\hat{n})|\c{H}|}{\hat{m}}\cdot\l(\frac{\tau(\hat{n},\hat{m})}{\hat{m}}\r)^{i-1}\cdot D^{|F|-i}\cdot \l(\frac{d}{2^{r+1}}\r)^{-|F|}\\
    &=\frac{\gamma(\hat{n})|\c{H}|}{D \hat{m}}\cdot\l(\frac{\tau(\hat{n},\hat{m})}{D \hat{m}}\r)^{i-1}\cdot \l(\frac{2^{r+1}D}{d}\r)^{|F|}.
\end{aligned}
$$
\end{proof}

In order to use Lemma~\ref{lemma:building from the shadow}, we will need to work with a hypergraph whose codegrees $\deg_H(\sigma)$ are well controlled.  We do this through the following regularization lemma.
\begin{lem}\label{lem:superregularize}
    If $r\ge 2$ and $H$ is an $n$-vertex $r$-partite $r$-graph on $V_1\cup \cdots \cup V_r$ with $n$ sufficiently large in terms of $r$, then there exists a subgraph $H'\sub H$ without isolated vertex and a real-valued vector $\Delta$ indexed by $2^{[r]}$ such that $|H'|\ge  (r\log n)^{-2^r}|H|/2$ and for any set $I\sub [r]$, every $|I|$-shadow $S$ of $H'$ in $\bigcup_{i\in I} V_i$ satisfies 
    $$
    \Delta_I (2r\log n)^{-2^r} \le \deg_{H'}(S)\le \Delta_I.
    $$

    Moreover, let $n'=|V(H')|$ and let $\partial_I$ denote the number of $|I|$-shadows of $H'$ in $\bigcup_{i\in I} V_i$. It is possible to relabel the indices of $V_1,\dots,V_r$ such that for any $1\le k\le r$,
    $$
    \partial_{[k]}\ge (n')^{\frac{r-k}{r-1}}|H'|^{\frac{k-1}{r-1}}(2r\log n)^{-k2^r},
    $$
    and
    $$
    \Delta_{[k]}\le \l(\frac{|H'|}{n'}\r)^{\frac{r-k}{r-1}}\l(2r(\log n)^{k+1}\r)^{2^r}.
    $$
\end{lem}
\begin{proof}
    Given an integral vector $y$ indexed by $2^{[r]}$, we say that an edge $e\in H$ is of \textit{type $y$} if for every set $I\sub [r]$, we have $2^{y_I-1}\le \deg_H(e\cap \bigcup_{i\in I} V_i)< 2^{y_I}$.  By the pigeonhole principle, there exists some $y\in [\log(n^r)]^{2^r}$ such that at least $\log(n^r)^{-2^r}|H| $ edges of $H$ are of type $y$.  Let $H_0$ be the collection of edges of type $y$, and let $\Delta$ be the vector with $\Delta_I:=2^{y_I}$.  
    
    Iteratively, given $H_0$, if there exists some set $I\sub [r]$ and $S\sub \bigcup_{i\in I} V_i$ with $|S|=|I|$ such that $0<\deg_{H_0}(S)< 2^{-2^{r}}\log(n^r)^{-2^r}\Delta_I$, then delete every edge from $H_0$ containing $S$.  Let $H'$ be the resulting hypergraph, noting that it satisfies the desired degree conditions by construction (since in particular, $\deg_{H'}(S)\le \deg_{H_0}(S)\le \Delta_I$ for any $S\sub \bigcup_{i\in I}V_i$).  Further, note  that for any set $I\sub [r]$, the number of sets $S\sub \bigcup_{i\in I} V_i$ of size $|I|$ with positive degree in $H_0$ is at most $2 |H|/\Delta_I$ (since any set with positive degree in $H_0$ has degree at least $\Delta_I/2$ in $H$ by definition of $H_0$), so the total number of edges deleted from $H_0$ to obtain $H'$ is at most 
    \[2^{-2^r}  \log(n^r)^{-2^r} \Delta_I\cdot 2|H|/\Delta_I= 2^{-2^r+1} \log(n^r)^{-2^r} |H|.\]  
    Summing this bound over all $I\sub [r]$ gives
    \[|H'|\ge |H_0|- 2^r\cdot 2^{-2^r+1} \log(n^r)^{-2^r} |H|\ge \log(n^r)^{-2^r} |H|- \frac{1}{2} \log(n^r)^{-2^r} |H|\ge \frac{1}{2} \log(n^r)^{-2^r} |H|,\]
    giving the desired lower bound on $|H'|$.  It thus remains to prove the ``moreover'' part of the statement.  For this we make use of the inequalities
    \begin{equation}\partial_I\Del_I(2r\log n)^{-2^r} \le |H'|\le \partial_I\Del_I,\label{eq:basic}\end{equation}
   which follow from $|H'|=\sum_{S\sub \bigcup_{i\in I}V_i} \deg_{H'}(S)$ for any $I$ and our bounds for these degrees .

    We may assume without loss of generality that $H'$ has no isolated vertices, as deleting these do not change any of the previously mentioned properties.  We will show by induction on $k$ that, upon relabeling indices of $V_i$, we have for all $1\le k\le r$ that
    $$
    \partial_{[k]}\ge (n')^{\frac{r-k}{r-1}}|H'|^{\frac{k-1}{r-1}}(2r\log n)^{-k2^r}.
    $$
    
    By the Pigeonhole Principle, we can assume without loss of generality that $\partial_{\{1\}}=|V(H')\cap V_1|\ge n'/r$. Let $2\le k\le r-1$ and suppose we have inductively shown
    \begin{equation}\label{eq:inductivePartial}
    \partial_{[k-1]}\ge (n')^{\frac{r-k+1}{r-1}}|H'|^{\frac{k-2}{r-1}}(2r\log n)^{-(k-1)2^r}.
    \end{equation}
    Assume for contradiction that $\partial_{[k-1]\cup\{j\}}< (n')^{\frac{r-k}{r-1}}|H'|^{\frac{k-1}{r-1}}(2r\log n)^{-k2^r}$ for every $k\le j \le r$.  In particular, this and \eqref{eq:basic} implies 
    \begin{equation}\label{eq:DelLower}
    \Delta_{[k]}\ge\frac{|H'|}{\partial_{[k]}}\ge\l(\frac{|H'|}{n'}\r)^{\frac{r-k}{r-1}}(2r\log n)^{k2^r}.
    \end{equation}
    \begin{claim}
        For every $k$-shadow $\{v_1,\ldots,v_k\}$ of $H'$ with $v_i\in V_i$ for all $i$, the number of vertices $u$ such that $\{v_1,\ldots,v_k,u\}$ is a $(k+1)$-shadow of $H'$ is at least
        \[\l(\frac{\Delta_{[k]}}{(2r\log n)^{2^r}}\r)^{\frac{1}{r-k}}.\]
    \end{claim}
    \begin{proof}
        Let $U$ denote the set of vertices $u$ as described in the claim.  Observe that every edge containing $\{v_1,\ldots,v_k\}$ in $H'$ consists of itself together with $r-k$ vertices of $U$, hence
        \[\deg_{H'}(v_1,\ldots,v_k)\le {|U|\choose r-k}\le |U|^{r-k},\]
        and rearranging gives \[|U|\ge \deg_{H'}(v_1,\ldots,v_k)^{\frac{1}{r-k}}\ge \left(\frac{\Del_{[k]}}{(2r \log n)^{2^r}}\right)^{\frac{1}{r-k}},\] with this last step using our lower bound for $I$-degrees that we verified in the first half of the proof.
    \end{proof}
    Observe that if $\{v_1,\ldots,v_k\}$ is a $k$-shadow of $H'$ as in the claim and $\{v_1,\ldots,v_k,u\}$ is a $(k+1)$-shadow, then $\{v_1,\ldots,v_{k-1},u\}$ is a $k$-shadow of $H'$ contained in $V_1\cup \cdots \cup V_{k-1}\cup V_j$ for some $k+1\le j\le r$.  Hence by the claim together with \eqref{eq:inductivePartial} and \eqref{eq:DelLower} we have
    
    \begin{align*}
    \sum^r_{j=k+1}\partial_{[k-1]\cup\{j\}}&\ge \partial_{[k-1]}\l(\frac{\Delta_{[k]}}{(2r\log n)^{2^r}}\r)^{\frac{1}{r-k}}\ge (n')^{\frac{r-k}{r-1}}|H'|^{\frac{k-1}{r-1}}(2r\log n)^{-(k-1)2^r-\frac{1}{r-k}2^r} \\ &\ge (r-k)(n')^{\frac{r-k}{r-1}}|H'|^{\frac{k-1}{r-1}}(2r\log n)^{-k2^r},
    \end{align*}
    where this last inequality holds with equality if $k=r-1$ and otherwise holds for $n$ sufficiently large (since the exponent for the logarithmic factor is strictly more than $-k 2^r$ in this case).
    
    By the Pigeonhole Principle, we conclude that for some $k+1\le j\le r$, 
    $$
   \partial_{[k-1]\cup\{j\}}\ge (n')^{\frac{r-k}{r-1}}|H'|^{\frac{k-1}{r-1}}(\log n)^{-k2^r}.
    $$
    This contradicts our assumption that $\partial_{[k-1]\cup\{j\}}< n'^{\frac{r-k}{r-1}}|H'|^{\frac{k-1}{r-1}}(2r\log n)^{-k2^r}$ for each $k\le j \le r$. Hence, by relabeling indices if necessary, we have
    $$
    \partial_{[k]}\ge (n')^{\frac{r-k}{r-1}}|H'|^{\frac{k-1}{r-1}}(2r\log n)^{-k2^r},
    $$
    completing our inductive lower bound for $\partial_{[k]}$.  
    Finally, by \eqref{eq:basic}  we have  for all $1\le k\le r$ that
    $$
    \Delta_{[k]}\le \frac{(2r\log n)^{2^r}|H'|}{\partial_{[k]}}\le \l(\frac{|H'|}{n'}\r)^{\frac{r-k}{r-1}}\l(2r(\log n)^{k+1}\r)^{2^r}.
    $$
\end{proof}

We now have all the tools we need to prove our first main result \Cref{lemma:BSviaSHADOW}, which we restate below for convenience. 

\begingroup
\def\thethm{\ref{lemma:BSviaSHADOW}}
\begin{thm}
Let $r>r_0\ge 2$ and let $F$ be an $r_0$-graph such that there exist positive functions $M_{r_0}=M_{r_0}(n)$, $\gamma_{r_0}=\gamma_{r_0}(n)$ and $\tau_{r_0}=\tau_{r_0}(n,m)$ with the following properties:
\begin{itemize}
    \item [(a)] $F$ is $(M_{r_0}, \gamma_{r_0}, \tau_{r_0})$-balanced,
    \item[(b)] For sufficiently large $n$, the function $M_{r_0}(n)n^{-\frac{r-r_0}{r-1}}$ is non-decreasing with respect to $n$,
    \item[(c)] For sufficiently large $n$, the function  $\gamma_{r_0}(n)$ is non-decreasing with respect to $n$.
    \item [(d)] For sufficiently large $n$,  $m\ge M_{r_0}(n)$, and any real $x\ge 1$, the function $\tau_{r_0}(n,m)$ is non-increasing with respect to $m$ and $\tau_{r_0}(nx,mx^{\frac{r-r_0}{r-1}})$ is non-decreasing with respect to $x$.
\end{itemize}
Then there exists a sufficiently large constant $C$ depending only on $F$ and $r$ such that for the functions
\begin{equation*}
    M_r(n):=\max\l\{M_{r_0}(n)^{\frac{r-1}{r_0-1}}n^{-\frac{r-r_0}{r_0-1}},~n^{r-1}\r\}(\log n)^{C},
\end{equation*}

\begin{equation*}
\gamma_r(n):=\gamma_{r_0}(n)(\log n)^{C},
\end{equation*}
and
\begin{equation*}
\tau_r(n,m):=\tau_{r_0}\l(n,n^{\frac{r-r_0}{r-1}}m^{\frac{r_0-1}{r-1}}(\log n)^{-C}\r)(\log n)^{C},
\end{equation*}
we have that $F^{(r)}$ is $(M_r,\gamma_r,\tau_r)$-balanced.
\end{thm}
\addtocounter{thm}{-1}
\endgroup

\begin{proof}
Let $H$ be an $n$-vertex $m$-edge $r$-graph where $n$ is sufficiently large and $m\ge M_r(n)$. Let $H'$ be a maximum $r$-partite subgraph of $H$, noting that $|H'|\ge \frac{r\,!}{r^r}\cdot m$ by considering a random $r$-partition of $H$. By Lemma~\ref{lem:superregularize} there exists a vector $\Delta$ and a subgraph $H''$ of $H'$ on $n''$ vertices with $r$-partition $V_1\cup\dots\cup V_r$ such that 
\begin{equation}\label{eq:LowerBoundOfH''}
|H''|\ge 2^{-1}(r\log n)^{-2^r}|H'|\ge\frac{r\,!\,m}{2r^r(r\log n)^{2^r}}\ge\frac{m}{(\log n)^{2^r+1}}, 
\end{equation}
with this last bound holding for $n$ sufficiently large; and moreover, every $r_0$-shadow $S$ of $H''$ in $\bigcup_{i\in [r_0]} V_i$ satisfies 
\begin{equation}\label{equation:codegree}
    (2r\log n)^{-2^r} \Delta_{[r_0]}\le \deg_{H''}(S)\le \Delta_{[r_0]},
\end{equation}
where
\begin{equation}\label{equation:UpperBoundOfDeltar0}
    \Delta_{[r_0]}\le \l(\frac{|H''|}{n''}\r)^{\frac{r-r_0}{r-1}}\l(2r(\log n)^{r_0+1}\r)^{2^r}\le \l(\frac{|H''|}{n''}\r)^{\frac{r-r_0}{r-1}}(\log n)^{r_02^{r+1}}.
\end{equation}

Let $H_{[r_0]}$ be the $r_0$-graph on $\bigcup_{i\in[r_0]}V_i$ whose edges are all $r_0$-shadows of edges of $H''$ in $\bigcup_{i\in[r_0]}V_i$.  Let $\hat{n}=|\bigcup_{i\in[r_0]}V_i|$ and let $\hat{m}=|H_{[r_0]}|$.  Our aim now will be to apply \Cref{lemma:building from the shadow} to $H_{[r_0]}$, and to this end we need to effectively lower bound $\hat{m}$ and, in view of \eqref{equation:codegree}, we will also need to show that $d:=\Del_{[r_0]}(2r\log n)^{-2^r}$ is at least $2v(F^{(r)}) n^{r-r_0-1}$.   Towards both of these goals, we observe by (\ref{equation:codegree}) that
\begin{equation}\label{equation:UpperBoundOfH''}
|H''|\le\Delta_{[r_0]}\hat{m}.
\end{equation}

By (\ref{equation:UpperBoundOfH''}) and (\ref{eq:LowerBoundOfH''}), we have
\begin{equation*}
\Delta_{[r_0]}\ge\frac{|H''|}{\hat{m}}\ge \frac{m}{\hat{m}(\log n)^{2^r+1}}.
\end{equation*}
Further, using the trivial bounds $\hat{m}\le n^{r_0}$ and $m\ge M_r(n)\ge n^{r-1}(\log n)^{2^{r+1}+2}$ for $C$ sufficiently large, we have
\begin{equation}\label{eq:Delta_r0}
(2r\log n)^{-2^r}\Delta_{[r_0]}\ge 2v(F^{(r)})n^{r-r_0-1}.
\end{equation}

Turning to $\hat{m}$, we have by (\ref{equation:UpperBoundOfH''}), (\ref{equation:UpperBoundOfDeltar0}), and (\ref{eq:LowerBoundOfH''}) that 
\begin{equation}\label{eq:partial}
    \hat{m}\ge\frac{|H''|}{\Delta_{[r_0]}}\ge (n'')^{\frac{r-r_0}{r-1}}|H''|^{\frac{r_0-1}{r-1}}(\log n)^{-r_02^{r+1}}\ge (n'')^{\frac{r-r_0}{r-1}}\l(\frac{m}{(\log n)^{2^r+1}}\r)^{\frac{r_0-1}{r-1}}(\log n)^{-r_02^r}.
\end{equation}

Using (\ref{eq:partial}) together with $m\ge M_r(n)$ and property (b), we have for $C$ sufficiently large that
$$
\begin{aligned}
\hat{m}&\ge (n'')^{\frac{r-r_0}{r-1}}\l(\frac{M_{r_0}(n)^{\frac{r-1}{r_0-1}}n^{-\frac{r-r_0}{r_0-1}}(\log n)^{C}}{(\log n)^{2^r+1}}\r)^{\frac{r_0-1}{r-1}}(\log n)^{-r_02^r}\\
&\ge (n'')^{\frac{r-r_0}{r-1}}\l(M_{r_0}(n'')^{\frac{r-1}{r_0-1}}(n'')^{-\frac{r-r_0}{r_0-1}}\r)^{\frac{r_0-1}{r-1}}\\
&=M_{r_0}(n'')\\
&\ge M_{r_0}(\hat{n}).
\end{aligned}
$$

Observe that we may assume $\hat{n}$ is sufficiently large.  Indeed, we assume in the statement that $n$ is sufficiently large, which combined with $m\ge M_r(n)\ge n^{r-1}$ and \eqref{eq:partial} implies $\hat{m}=|H_{[r_0]}|$ is sufficiently large, from which it follows that $\hat{n}=v(H_{[r_0]})$ is sufficiently large. Using this  and property (a), we may apply Lemma~\ref{lemma:building from the shadow} to $H''$, giving a collection $\c{H}$ of copies of $F^{(r)}$ such that for all $ 1\le i\le |F|$,
\begin{equation}
    \Delta_i(\c{H})\le \frac{\gamma_{r_0}(\hat{n})|\c{H}|}{\Delta_{r_0}\hat{m}}\cdot\l(\frac{\tau_{r_0}(\hat{n},\hat{m})}{\Delta_{r_0}\hat{m}}\r)^{i-1}\l(2^{r+1}(2r\log n)^{2^r}\r)^{|F|}.
\end{equation}

Using this with (\ref{eq:LowerBoundOfH''}) and (\ref{equation:UpperBoundOfH''}) gives,

\begin{equation}\label{eq:UpperboundOfDelta}
\begin{aligned}
    \Delta_i(\c{H})&\le\frac{\l(2^{r+1}(2r\log n)^{2^r}\r)^{|F|}\cdot\gamma_{r_0}(\hat{n})|\c{H}|}{|H''|}\cdot\l(\frac{\tau_{r_0}(\hat{n},\hat{m})}{|H''|}\r)^{i-1}\\
    &\le\frac{2r^r(r\log n)^{2^r}}{r\,!}\cdot\frac{\l(2^{r+1}(2r\log n)^{2^r}\r)^{|F|}\cdot\gamma_{r_0}(n)|\c{H}|}{m}\cdot\l(\frac{2r^r(r\log n)^{2^r}}{r\,!}\cdot\frac{\tau_{r_0}(\hat{n},\hat{m})}{m}\r)^{i-1}.
\end{aligned}
\end{equation}
Note that, by both parts of property (d) and by ({\ref{eq:partial}}),
\begin{equation}\label{eq:tau}
\begin{aligned}
    \tau_{r_0}(\hat{n},\hat{m})&\le \tau_{r_0}\l(\hat{n}, \hat{n}^{\frac{r-r_0}{r-1}}\l(\frac{r\,!\,m}{2r^r(r\log n)^{2^r}}\r)^{\frac{r_0-1}{r-1}}(\log n)^{-r_02^r}\r)\\
    &\le  \tau_{r_0}\l(n, n^{\frac{r-r_0}{r-1}}\l(\frac{r\,!\,m}{2r^r(r\log n)^{2^r}}\r)^{\frac{r_0-1}{r-1}}(\log n)^{-r_02^r}\r).
\end{aligned}
\end{equation}

Finally, by property (c), the definition of $\gamma_r$ and $\tau_r$, and (\ref{eq:UpperboundOfDelta}) and (\ref{eq:tau}), we conclude that for all $1\le i\le |F|$,
$$
\Delta_i(\c{H})\le \frac{\gamma_r(n)|\c{H}|}{m}\cdot\l(\frac{\tau_r(n,m)}{m}\r)^{i-1}.
$$
Therefore, $F^{(r)}$ is $(M_r,\gamma_r,\tau_r)$-balanced for $n$ sufficiently large by definition.
\end{proof}

\subsection{Greedy Expansion: Proof of \Cref{lemma:BSviaGreedy}}
We now prove our second result \Cref{lemma:BSviaGreedy} which lifts balancedness from $(r-1)$-graphs $F$ to $F^{(r)}$ whenever $F^{(r)}$ is contained in a tight $r$-tree.  We begin by introducing two  lemmas that will be useful. The following says that every $r$-partite $r$-graph contains a large subgraph such that either all $(r-1)$-shadows have large codegrees or all $(r-1)$-shadows in some $(r-1)$ parts have small (but not too small) and regular codegrees.

\begin{lem}\label{lemma:regularize2}
    For $r\ge 2$, let $H$ be an $r$-partite $r$-graph on $n$ vertices, then for any $|H|/(4n^{r-1})< A\le n$, one of the following two statements is true: 
    \begin{itemize}
        \item [(i)] There exists a subgraph $\hat{H}\subseteq H$ with $|\hat{H}|\ge |H|/2$ such that every $(r-1)$-shadow of $\hat{H}$ has codegree at least $A$.
        \item[(ii)] There  exist a subgraph $\hat{H}\subseteq H$ with $|\hat{H}|\ge |H|/(4r\log n)$, an $r$-partition $V(\hat{H})=V_1\cup\dots\cup V_r$, and
        $$
        \frac{|H|}{4n^{r-1}}<D\le A
        $$
        such that any $(r-1)$-shadow $\sigma$ of $\hat{H}$ in $V_1\cup\dots\cup V_{r-1}$ satisfies 
        $$
        D/2\le \deg_{\hat{H}}(\sigma)<D.
        $$

    \end{itemize}
\end{lem}

\begin{proof}
Let $V(H)=V_1\cup\dots\cup V_r$ be an $r$-partition of $H$. We run the following algorithm that partitions $H$ into subgraphs $\Hh_L$ and $\Hh_{J, a}$, where $J\in \binom{[r]}{r-1}$ and $1\le a\le \log A$:
    
    \begin{itemize}
        \item[1.] Let $H_0=H$. Given $H_i$ for some $i\ge 0$, if there is an $(r-1)$-shadow $\sigma$ of $H_{i}$ such that $\deg_{H_{i}}(\sigma)\le A$, then we arbitrarily fix such a $\sigma$ and let $H_{i+1}=H_{i}-N_{H_{i}}(\sigma)$, where $N_{H_i}(\sigma)$ denotes the set of edges in $H_{i}$ containing $\sigma$. Let $(J,a)$ be the unique pair such that $\sigma$ contain a vertex of $V_j$ for each $j\in J$ and $2^{a-1}\le \deg_{H_{i}}(\sigma)<2^{a}$. We put all the elements in $N_{H_{i}}(\sigma)$ into the edge set of $\Hh_{J, a}$. 
        \item[2.] If all $(r-1)$-shadows $\sigma$ of $H_{i}$ have $\deg_{H_{i}}(\sigma)> A$, then we let $\Hh_L=H_{i}$ and terminate the algorithm.
    \end{itemize}

If $|\Hh_L|\ge |H|/2$, we let $\hat{H}=\Hh_L$, which by definition implies $\hat{H}$ satisfies (i). Otherwise,
$$
\sum_{1\le a\le \log A}\sum_{J\in\binom{[r]}{r-1}}|\Hh_{J, a}|>|H|/2.
$$

Note that each $(r-1)$-shadow $\sigma$ of $H$ contributes at most $2^a$ to some $|\Hh_{J, a}|$, and the number of $(r-1)$-shadows is at most $n^{r-1}$. Thus we have 
$$
\sum_{1\le a\le \log \frac{|H|}{4n^{r-1}}}\sum_{J\in\binom{[r]}{r-1}}|\Hh_{J, a}|\le n^{r-1}\frac{|H|}{4n^{r-1}}=\frac{|H|}{4},
$$
and hence
$$
\sum_{\log \frac{|H|}{4n^{r-1}}< a\le \log A}\sum_{J\in\binom{[r]}{r-1}}|\Hh_{J, a}|>\frac{|H|}{4}.
$$

By the Pigeonhole Principle, there exists a pair $(J,a)$ such that $2^a\ge \frac{|H|}{4n^{r-1}}$ and that $|\Hh_{J,a}|\ge |H|/(4r\log n)$. Fix such a pair $(J,a)$ and let $\hat{H}=\Hh_{J,a}$. Without loss of generality, we can suppose $J=\{1,\dots,r-1\}$. By definition, any $(r-1)$-shadow $\sigma$ in $V_1\cup\dots\cup V_{r-1}$ has $2^{a-1}\le \deg_{\hat{H}}(\sigma)<\min\{2^{a},A\}$. Therefore, $\hat{H}$ satisfies (ii) by letting $D=\min\{2^{a},A\}\ge |H|/(4n^{r-1})$.
\end{proof}

The following lemma comes from~\cite[Lemma 2.2]{nie2023random} and is proven in essentially the same way as described in the second approach of \Cref{subsec:sketch}.
\begin{lem}[\cite{nie2023random} Lemma 2.2]\label{lemma:bss_sub_tree}
    Let $T$ be a tight $r$-tree and let $F$ be a spanning subgraph of $T$ with $|F|\ge 2$. There exists a constant $C>0$ such that the following holds for all sufficiently large $n$ and every $A$ with $C\le A\le \binom{n}{r}/n^{r-1}$. Let $H$ be an $r$-graph on $n$ vertices. If every $(r-1)$-shadow $\sigma$ of $H$ has $\deg_H(\sigma)> A$, then there exists a collection $\c{H}$ of copies of $F$ in $H$ satisfying
    \begin{enumerate}
        \item[(a)] $|\c{H}|\ge C^{-1} |H|A^{v(F)-r}$;
        \item[(b)] $\Delta_i(\c{H})\le CA^{v(F)-r-\frac{i-1}{d_r(F)}}$,~~$1\le i\le |F|$.
    \end{enumerate}
    In particular, we have for all $1\le i\le |F|$ that
    \[\Del_i(\c{H})\le C^2 \frac{|\c{H}|}{|H|} A^{-\frac{i-1}{d_r(F)}}.\]
\end{lem}

We are now ready to prove \Cref{lemma:BSviaGreedy}, which we restate below for convenience.
\begingroup
\def\thethm{\ref{lemma:BSviaGreedy}}
\begin{thm}
Let $F$ be an $(r-1)$-graph with $r-1,|F|\ge 2$ such that $F^{(r)}$ is a spanning subgraph of a tight $r$-tree and such that there exist positive functions $M_{r-1}=M_{r-1}(n)$, $\gamma_{r-1}=\gamma_{r-1}(n)$ and $\tau_{r-1}=\tau_{r-1}(n,m)$ with the following properties: 
\begin{itemize}
    \item[(a)] $F$ is $(M_{r-1},\gamma_{r-1},\tau_{r-1})$-balanced,
    \item[(b)] For sufficiently large $n$ the function $\gamma_{r-1}(n)$ is non-decreasing with respect to $n$,
    \item[(c)] For sufficiently large $n$ and $m\ge M_{r-1}(n)$, the function  $\tau_{r-1}(n,m)$ is non-decreasing with respect to $n$ and non-increasing with respect to $m$.
    
\end{itemize}
Then there exists a sufficiently large constant $C$ depending only on $F$  such that the following holds. Let 
\[M_r(n):=Cn^{r-1},\]
\begin{equation*}
\gamma_{r}(n):=\gamma_{r-1}(n)\cdot C\log n,  
\end{equation*}
and let $A=A(n,m)$ be any function such that, for all sufficiently large $n$ and $m\ge M_r(n)$,
\begin{itemize}
    \item[(d)]
    \begin{equation*}
        m n^{1-r}\le A(n,m)\le \frac{m}{M_{r-1}(n)\cdot C\log n},
    \end{equation*}
\end{itemize}
 
and let
\begin{equation*}
\tau_{r}(n,m):=\max\l\{A(n,m)^{-\frac{1}{d_r\l(F^{(r)}\r)}}\cdot m,~\tau_{r-1}\l(n,~\frac{m}{A(n,m)\cdot C\log n}\r)\cdot C\log n\r\}.   
\end{equation*}
Then $F^{(r)}$ is $(M_r,\gamma_{r},\tau_{r})$-balanced.
\end{thm}
\addtocounter{thm}{-1}
\endgroup

\begin{proof}
Let $H$ be an $n$-vertex $m$-edge $r$-graph such that $n$ is sufficiently large and $m\ge M_{r}(n)$. Let $H'\sub H$ be a maximum sized $r$-partite subgraph of $H$. By considering a random $r$-partition of $V(H)$, we know that $|H'|\ge \frac{r\,!}{r^{r}}m$. By Lemma~\ref{lemma:regularize2}, there exists a subgraph $\hat{H}\subseteq H'$ with 
\begin{equation}\label{equation:greedy_H''LowerBound}
|\hat{H}|\ge \frac{|H'|}{4r\log n}\ge\frac{r\,!}{4r^{r+1}}\cdot\frac{m}{\log n}    
\end{equation}
and an $r$-partition $V(\hat{H})=V_1\cup\dots\cup V_{r}$ such that one of the following two statements is true.
\begin{itemize}
    \item[(i)]  Every $(r-1)$-shadow of $\hat{H}$ has codegree at least $A$.
    \item[(ii)] There exists $|H'|/(4n^{r-1})<D\le A$ such that every $(r-1)$-shadow $\sigma$ of $\hat{H}$ in $V_1\cup\dots\cup V_{r-1}$ satisfies 
        $$
        D/2\le \deg_{\hat{H}}(\sigma)<D.
        $$
\end{itemize}
The proof splits into two cases accordingly.

\noindent\textbf{Case 1: Statement (i) is true.}  Because $A\ge m n^{1-r}$ by property (d) and because $m\ge M_r(n)=C n^{r-1}$ with $C$ sufficiently large, we can apply Lemma~\ref{lemma:bss_sub_tree} to conclude that there exists a constant $C'$ and a collection $\c{H}$ of copies of $F^{(r)}$ in $\hat{H}$ such that for all $1\le i\le |F^{(r)}|$,
\begin{equation*}
\Delta_i(\c{H})\le C'\frac{|\c{H}|}{|\hat{H}|}A^{-\frac{i-1}{d_r(F^{(r)})}}\le \frac{C'2r^{r+1}\log n}{r\,!}\cdot\frac{|\c{H}|}{m}\l(\frac{A^{-\frac{1}{d_r(F^{(r)})}}\cdot m}{m}\r)^{i-1}\le \gamma_r(n)\cdot \frac{|\c{H}|}{m}\left(\frac{\tau_r(n,m)}{m}\right)^{i-1},
\end{equation*}
where the second inequality used \eqref{equation:greedy_H''LowerBound}, and  this last step used the definition of $\gamma_r$ and property (b) to conclude \[\gamma_r(n)\ge C \log n\cdot \min\{1,\gamma_{r-1}(n_{r-1})\}\ge \frac{C' 2r^{r+1}\log n}{r!}\]
provided $C$ is sufficiently large, and similarly that $\tau_r(n,m)\ge A^{-\frac{1}{d_r(F^{(r)})}}\cdot m$.

\noindent\textbf{Case 2: Statement (ii) is true.} We aim to apply Lemma~\ref{lemma:building from the shadow} on $\hat{H}$.  To this end, let $\hat{H}_{[r-1]}$ be the $(r-1)$-graph on $\cup_{i\in [r-1]} V_i$ whose edges are all $(r-1)$-shadows of $\hat{H}$ in $\cup_{i\in [r-1]} V_i$. Let $\hat{n}=v(\hat{H}_{[r-1]})=|\cup_{i\in [r-1]} V_i|$ and let $\hat{m}=|\hat{H}_{[r-1]}|$. 

Recall that $D\ge \frac{|H'|}{4n^{r-1}}\ge\frac{r\,!\,m}{4r^rn^{r-1}}$ and that $m\ge Cn^{r-1}$. So for $C$ sufficiently large, we have
$$
D/2\ge\frac{Cr\,!}{4r^r}\ge 2v(F^{(r)}).
$$
By statement (ii), we have 
\begin{equation}\label{equation:greedy_UpperBoundOfH''}
D\hat{m}\ge |\hat{H}|.
\end{equation}

Thus by (\ref{equation:greedy_UpperBoundOfH''}), (\ref{equation:greedy_H''LowerBound}), the hypothesis $D\le A$, and property (d), we have
$$
\hat{m}\ge \frac{|\hat{H}|}{D}\ge \frac{r\,!}{4r^{r+1}}\cdot\frac{m}{A\cdot \log n}\ge M_{r-1}(n).
$$

Given sufficiently large $n$, we have sufficiently large $\hat{m}$, and hence sufficiently large $\hat{n}$. So by property (a) and the inequalities above, we can apply Lemma~\ref{lemma:building from the shadow} on $\hat{H}$ to obtain a collection $\c{H}$ of copies of $F^{(r)}$ such that, for all $1\le i\le |F|$,
\begin{equation*}
\Delta_i(\c{H})\le \frac{\gamma_{r-1}(\hat{n})|\c{H}|}{D\hat{m}}\cdot\l(\frac{\tau_{r-1}(\hat{n},\hat{m})}{D\hat{m}}\r)^{i-1}4^{r|F|}.
\end{equation*}
By properties (b), (c) and (d), and inequalities (\ref{equation:greedy_H''LowerBound}) and (\ref{equation:greedy_UpperBoundOfH''}), 
\begin{align*}
\Delta_i(\c{H})&\le \frac{4^{r|F|}2r^{r+1}\log n}{r\,!}\cdot \frac{\gamma_{r-1}(n)|\c{H}|}{m}\cdot\l(\frac{2r^{r+1}\log n}{r\,!}\cdot\frac{\tau_{r-1}\l(n,\frac{r\,!\cdot m}{2r^{r+1}\log n\cdot A}\r)}{m}\r)^{i-1}\\ &\le \frac{\gamma_{r}(n)|\c{H}|}{m}\l(\frac{\tau_{r}(n,m)}{m}\r)^{i-1}.
\end{align*}

In total then, we find that in either case there exists a collection $\c{H}$ satisfying, for all $1\le i\le |F|$ that
\[\Del_i(\c{H})\le \frac{\gamma_{r}(n)|\c{H}|}{m}\l(\frac{\tau_{r}(n,m)}{m}\r)^{i-1},\]
showing that  $F^{(r)}$ is indeed $(M_r,\gamma_r,\tau_r)$-balanced.
\end{proof}

\section{Optimal Balanced Supersaturation}\label{sec:optimal}
Here we explore what can be said about (expansions) $F$ which have ``optimal'' balanced supersaturation.  More precisely, we observe the following limits for how balanced a hypergraph $F$ can be.

\begin{prop}\label{prop:optimalBSBound}
If $F$ is an $r$-graph with $r,|F|\ge 2$ which is $(M,\gamma,\tau)$-balanced, then for sufficiently large $n$ we have $M(n)> \ex(n,F)$ and 
$$
\tau(n,m)\ge r^{-\frac{r|F|}{|F|-1}}\gamma(n)^{-\frac{1}{|F|-1}}n^{r-\frac{1}{d_r(F)}}.
$$ 
\end{prop}

\begin{proof}
If $M(n)\le \ex(n,F)$ for infinitely many $n$, then one could consider an infinite sequence of $F$-free $r$-graphs $H$ with $n$ vertices and $\ex(n,F)$ edges, from which no collection $\c{H}$ of copies of $F$ can be guaranteed, contradicting $F$ being $(M,\gamma,\tau)$-balanced.

Finally, let $H$ be the random $r$-uniform hypergraph $G^r_{n,p}$ with $p=m/\binom{n}{r}$ for some $m\ge M(n)$. With positive probability we have $|H|\ge m$ and the number of copies of $F$ in $H$ is at most $p^{|F|}n^{v(F)}\le r^{r|F|}m^{|F|}n^{v(F)-r|F|}$. Let $\c{H}$ be a non-empty collection of copies of $F$ in $H$ satisfying (\ref{equation:Delta}). Then we have $|\c{H}|\le r^{r|F|}m^{|F|}n^{v(F)-r|F|}$. Note that $\Delta_{|F|}(\c{H})=1$, so by (\ref{equation:Delta}) we have 
$$
\tau(n,m)\ge \l(\frac{m}{\gamma(n)|\c{H}|}\r)^{\frac{1}{|F|-1}}m\ge r^{-\frac{r|F|}{|F|-1}}n^{r-\frac{v(F)-r}{|F|-1}}\gamma(n)^{-\frac{1}{|F|-1}} \ge r^{-\frac{r|F|}{|F|-1}}n^{r-\frac{1}{d_r(F)}}\gamma(n)^{-\frac{1}{|F|-1}}.
$$
\end{proof}

Our main goal of this section is to prove a result (\Cref{prop:optimalBSExpansions}) giving a sufficient condition for expansions to have optimal balanced supersaturation, which will be the main ingredient for proving \Cref{thm:optimalBalanced}.  We will need two lemmas regarding expansions, the first of which shows that $r$-densities play nicely with expansions.

\begin{lem}\label{lem:densityRelation}
    For any $r_0$-graph $F$ with $r_0,|F|\ge 2$ and integer $r\ge r_0$, we have 
    \[\frac{1}{d_r(F^{(r)})}= r-r_0+\frac{1}{d_{r_0}(F)},\]
    and in particular $d_r(F^{(r)})<\frac{1}{r-r_0}$.
\end{lem}
\begin{proof}
    Let $F'\sub F$ be a subgraph satisfying $d_{r_0}(F)=\frac{|F'|-1}{v(F')-r_0}$.  Then the expansion $F'':=(F')^{(r)}\sub F^{(r)}$ satisfies
    \[\frac{1}{d_r(F^{(r)})}\le \frac{v(F'')-r}{|F''|-1}=\frac{v(F')+(r-r_0)|F'|-r}{|F'|-1}=r-r_0+\frac{1}{d_{r_0}(F)}.\]

    Now let $F'\sub F^{(r)}$ be the subgraph satisfying $d_r(F^{(r)})=\frac{|F'|-1}{v(F')-r}$.  Define $F''\sub F$ to be the induced subraph on $V(F')\cap V(F)$.  Observe that $|F''|\ge |F'|$ and $v(F'')\le v(F')-(r-r_0)|F'|$ (since each edge of $F'$ contains $r-r_0$ vertices not in $V(F)$).  This implies that
    \[\frac{1}{d_{r_0}(F)}\le \frac{v(F'')-r_0}{|F''|-1}\le \frac{v(F')-r_0-(r-r_0)|F'|}{|F'|-1}=\frac{1}{d_r(F^{(r)})}-r+r_0,\]
    which together with the previous inequality gives the desired result.
\end{proof}

We next have a lemma which gives sufficient conditions for (expansions of) $F$ to be a spanning subgraph of a tight tree.  Here the $k$-shadow of an $r_0$-graph $T$ is defined to be the $k$-graph on $V(T)$ consisting of the $k$-shadows of $T$.

\begin{lem}\label{lem:spanningTreeInduction}
    If $F$ is a $k$-graph and $r_0\ge k$ an integer such that there exists a tight $r_0$-tree $T'$ whose $k$-shadow contains $F$ as a spanning subgraph, then for all $r\ge r_0$ there exists a tight $r$-tree $T$ which contains $F^{(r)}$ as a spanning subgraph.
\end{lem}
In particular, if $r_0=k$ then the hypothesis of this lemma simply says that $F$ is a spanning subgraph of a tight $r_0$-tree.
\begin{proof}
    The result is trivial if $F$ is a single edge, so we will assume this is not the case.  Given a pair of integers $r\ge r_0$ and $\ell\ge k$, we will say that $F$ is $(r,\ell)$-good if there exists some ordering $v_1,v_2,\ldots,v_t$ of the vertices of $F^{(\ell)}$ and an ordered set of $r$-sets $h_r,h_{r+1}\ldots,h_t$ such that (1) $h_{r}=\{v_1,\ldots,v_{r}\}$, (2) for all $i>r$ we have $e_i\setminus \bigcup_{j<i} h_j=\{v_{i}\}$ and there exists some $i'<i$ with $h_{i}-v_{i}\sub h_{i'}$, (3) every edge of $F^{(\ell)}$ is contained in some $h_i$.

    We first observe that $F$ being $(r,\ell)$-good implies that the $r$-graph $T$ with edges $h_r,\ldots,h_t$ is a tight $r$-tree whose $\ell$-shadow contains $F^{(\ell)}$ as a spanning subgraph.  In particular, the hypothesis  implies that $F$ is $(r_0,k)$-good.  We aim to show that $F$ is $(r,r)$-good for any $r\ge r_0$, proving the result.  We do this through double induction.

    Assume that we shown $F$ to be $(r,\ell)$-good for some $r\ge r_0$ and $r> \ell\ge k$; we aim to show that $F$ is  $(r,\ell+1)$-good.  Let $h_r,\ldots,h_t$ be the ordered $r$-sets guaranteed by $F$ being $(r,\ell)$-good.  For each edge $e\in F^{(\ell)}$, let $v_e$ denote the unique vertex such that $\{v_e\}\cup e$ is an edge of $F^{(\ell+1)}$, and let $i_e$ denote any index such that $h_{i_e}$ contains $e$ (which exists by (3)).  For each $e\in F^{(\ell)}$, we let $h_e$ be any $r$-set which consists of $v_e$ and any $(r-1)$-subset of $h_{i_e}$ which contains $e$ (which exists since $|e|=\ell<r$ by hypothesis).  It is not difficult to check that $h_r,\ldots,h_t$ followed by any ordering of these $h_e$ sets verifies that $F$ is $(r,\ell+1)$-good, proving this claim.

    Now assume we have shown $F$ to be $(r,r)$-good for some $r\ge r_0,k$.   We aim to show that $F$ is $(r+1,r)$-good.  Let $h_r,\ldots,h_t$ be the ordered $r$-sets guaranteed by $F$ being $(r,\ell)$-good.  Note that we must have $r<t$, as otherwise (3) would imply that $F^{(r)}$ must in fact just be a single edge, which we assumed not to be the case.  With this in mind, we construct $(r+1)$-sets $h_{r+1}',\ldots,h_t'$ by having $h_i'=h_{i'}\cup \{v_i\}$ for all $i\ge r+1$.  It is not difficult to see that these $(r+1)$-sets verify $F$ as being $(r+1,r)$-good.

    Combining the two claims above together with the observation that $F$ is $(r_0,\ell)$-good gives that $F$ is $(r,r)$-good for all $r\ge r_0$, proving the result.
\end{proof}

We now prove our main result of the section, which roughly shows that if $F$ has sufficiently high uniformity and optimal balanced supersaturation, then so does every expansion of $F$.

\begin{prop}\label{prop:optimalBSExpansions}
Let $F$ be an $r_0$-graph with $\Delta(F)\ge 2$ which is the spanning subgraph of a tight $r_0$-tree such that there exists a constant $C_{r_0}$ such that $F$ is $(M_{r_0},\gamma_{r_0},\tau_{r_0})$-balanced with $M_{r_0}(n)= C_{r_0}n^{r_0-1}$, $\gamma_{r_0}(n)= (\log n)^{C_{r_0}}$ and $\tau_{r_0}(n,m)= n^{r_0-\frac{1}{d_{r_0}(F)}}(\log n)^{C_{r_0}}$.  Then for all $r\ge r_0$, $F^{(r)}$ is the spanning subgraph of a tight tree and there exists a constant $C_r$ such that $F^{(r)}$ is $(M_r,\gamma_r,\tau_r)$-balanced with $M_r(n)= C_r n^{r-1}$, $\gamma_r(n)= (\log n)^{C_r}$ and $\tau_r(n,m)=n^{r-\frac{1}{d_r(F)}}(\log n)^{C_r}$.
\end{prop}
\begin{proof}
We claim that it suffices to prove the result when $r_0=r-1$.  Indeed, assuming this case we prove the general result by induction on $r$, the $r=r_0$ case being trivial.  Assume that we have proven the result up to some $r>r_0$.  Because the result holds at $r-1$, $F^{(r-1)}$ continues to satisfy the hypothesis of the lemma, so the assumed $r-1$ result completes the induction.  With this we may assume that $r_0=r-1$ from now on.  

It is straightforward to check that properties (a), (b), and (c) of  \Cref{lemma:BSviaGreedy} are satisfied, and we have that $F^{(r)}$ is a spanning subgraph of a tight tree due to \Cref{lem:spanningTreeInduction} applied with $k=r_0$.  Let $C$ be the constant guaranteed to exist from \Cref{lemma:BSviaGreedy}.   Let $M_r(n)=Cn^{r-1}$,  let $\gamma_r(n)=\gamma_{r-1}(n)\cdot C\log n= C(\log n)^{C_{r-1}+1}$, and let \[A(n,m)=\max\{m^{d_{r}(F^{(r)})}n^{-rd_{r}(F^{(r)})+1},~ m n^{1-r}\}.\] 
We claim that property (d) of \Cref{lemma:BSviaGreedy} holds.  Indeed, the lower bound $A(n,m)\ge m n^{1-r}$ holds trivially since this is one of the terms in the maximum.   The upper bound $A(n,m)\le \frac{m}{M_{r-1}(n) C\log n}=\frac{m}{C_{r-1} n^{r-2}\cdot C \log n}$ holds for the second term of $A$ provided ${n\choose r-1}\ge C_{r-1} n^{r-2}$, which holds for $n$ sufficiently large, and it holds for the first term provided \[m^{d_{r}(F^{(r)})-1}n^{-rd_{r}(F^{(r)})+1}\le \frac{1}{C_{r-1}\cdot C \log n} n^{2-r}.\] 
Because $d_r(F^{(r)})<1$ by \Cref{lem:densityRelation}, this inequality is hardest to satisfy at the smallest possible value of $ m\ge C n^{r-1}$, and one can verify this holds for $n$ sufficiently large through the inequality $-d_r(F^{(r)})+2-r<2-r$.

Define
$$
\begin{aligned}
\tau_r(n,m)&=\max\l\{A(n,m)^{-\frac{1}{d_r\l(F^{(r)}\r)}}\cdot m,~\tau_{r-1}\l(n,~\frac{m}{A(n,m)\cdot C\log n}\r)\cdot C\log n\r\}\\
&\le n^{r-\frac{1}{d_r\l(F^{(r)}\r)}}\cdot C(\log n)^{C_{r-1}+1},
\end{aligned}
$$
where this last step used $\tau_{r-1}=n^{r-1-\frac{1}{d_{r-1}(F)}} (\log n)^{C_{r-1}}=n^{r-\frac{1}{d_r(F^{(r)}}}(\log n)^{C_{r-1}}$ by \Cref{lem:densityRelation}.  We conclude by \Cref{lemma:BSviaGreedy} that $F^{(r)}$ is $(M_r,\gamma_r,\tau_r)$-balanced.
\end{proof}








\section{General Random Tur\'an Results}\label{sec:Turan}

Here we prove a set of general lower and upper bounds for the random Tur\'an problem, special cases of which have appeared throughout the literature in previous papers.   We begin with proving the simple bounds of \Cref{lem:generalLowerRandomTuran}, which we restate here for the reader's convenience.

\begingroup
\def\thethm{\ref{lem:generalLowerRandomTuran}}
\begin{prop}
    Let $F$ be an $r$-graph with $\Delta(F)\ge 2$.  If $n^{-r}\ll p \ll n^{-\frac{1}{d_r(F)}}$, then a.a.s.
    \[\ex(G_{n,p}^r,F)=(1+o(1))p {n\choose r}.\]
    If $p\gg  n^{-\frac{1}{d_r(F)}}$, then a.a.s.
    \[\ex(G_{n,p}^r,F)\ge  \max\{\Omega(p\cdot \ex(n,F)),n^{r-\frac{1}{d_r(F)}}(\log n)^{-1}\}.\]
\end{prop}
\endgroup

\begin{proof}
    For the upper bound, we observe that trivially $\ex(G_{n,p}^r,F)\le |G_{n,p}^r|$, and this is at most $(1+o(1))p{n\choose r}$ a.a.s.\ provided $p\gg n^{-r}$ by the Chernoff bound. 
	
    For the first lower bound, let $F'\sub F$ be such that $d_r(F)=\f{|F'|-1}{v(F')-r}$, and let $X$ be the number of copies of $F'$ in $G_{n,p}^r$.  Note that $\E[X]=\Theta(p^{|F'|} n^{v(F')})=o(p {n\choose r})$ by our range of $p$.  As such we have
    $\omega:=\f{p{n\choose r}}{\E[X]}\to \infty.$  
    Choose $\ep=\ep(n)$ any function such that $\ep\to 0$ and $\ep \omega\to \infty$.  By Markov's inequality,
	\[\Pr\left[X \ge \ep p {n\choose r}\right]\le \rec{\ep \omega}\to 0.\]
	Thus a.a.s.\ $X=o(p{n\choose r})$.  By the Chernoff bound, $|G_{n,p}^r|\ge (1-o(1))p{n\choose r}$ a.a.s.\ for $p\gg n^{-r}$.  Thus by deleting an edge from each of the $X$ copies of $F$ in $G_{n,p}^r$, we obtain an $F$-free subgraph of $G_{n,p}^r$ with $(1-o(1))p{n\choose r}$ edges a.a.s.

    For the remaining lower bounds, we observe that $\Delta(F)\ge 2$ implies (and is in fact equivalent to) $d_r(F)>1/r$ by considering $F'\sub F$ consisting of two intersecting edges.  The lower bound $n^{r-\frac{1}{d_r(F)}}(\log n)^{-1}$ then follows from the monotonicity of $\ex(G_{n,p}^r,F)$ with respect to $p$: more precisely, taking $q:=n^{-\frac{1}{d_r(F)}}(\log\log n)^{-1}$, we see that $n^{-r}\ll q\ll n^{-\frac{1}{d_r(F)}}$ (because $d_r(F)>1/r$), and hence the bound above implies that a.a.s.\ $\ex(G_{n,q}^r,F)\ge \frac{1}{2} q{n\choose r}\ge n^{r-\frac{1}{d_r(F)}}(\log n)^{-1}$, and hence this same lower bound continues to hold a.a.s.\ for $\ex(G_{n,p}^r,F)$ for any $p\ge q$.
    
    Finally, the subgraph $G\sub G_{n,p}^r$ obtained by intersecting $G_{n,p}^r$ with an $F$-free subgraph with $\ex(n,F)$ edges is a binomial random variable with mean $p\cdot \ex(n,F)$, so this will be at least $\Omega(p\cdot \ex(n,F))$ a.a.s.\ provided $p\cdot \ex(n,F)\to \infty$.  And indeed, due to the maximum in the lower bound we only need to show this in the case when $p\cdot \ex(n,F)\ge n^{r-\frac{1}{d_r(F)}}(\log n)^{-1}\to \infty$ since $d_r(F)>1/r$.
\end{proof}

We now move onto our upper bound for the random Tur\'an problem \Cref{Lemma:General Random Turan}.  Our proof uses the hypergaph container method, developed independently by Balogh--Morris--Samotij~\cite{balogh2015independent} and Saxton--Thomason~\cite{saxton2015hypergraph}, which is fundamental in the study of random Turán problems.  Here we make use of the following simplified version of  Theorem 1.5 in~\cite{morris2024asymmetric}:
\begin{thm}[\cite{morris2024asymmetric}]\label{theorem:container}
For every integer $k \ge 2$, there exists a constant $\epsilon>0$ such that the following holds.
Let $B,L\ge 1$ be positive integers and $\c{H}$ a $k$-graph satisfying
    \begin{equation}\label{equation:container}
        \Delta_j(\c{H})\leq \l(\frac{B}{v(\c{H})}\r)^{j-1}\frac{|\c{H}|}{L},~~\forall 1\le j\le  k.
    \end{equation}
    
     Then there exists a collection $\mathcal{C}$ of subsets of $V(\c{H})$ such that:
    \begin{enumerate}
        \item[(i)] For every independent set $I$ of $\c{H}$, there exists $S\in \mathcal{C}$ such that $I\subset S$;
        \item[(ii)] For every $S\in\mathcal{C}$, $|S|\leq v(\c{H})-\epsilon L$;
        \item[(iii)] We have \[|\mathcal{C}|\le \exp\l(\frac{\log\l(v(\c{H})\r)B}{\epsilon}\r).\]
    \end{enumerate}
\end{thm}

We now proving \Cref{Lemma:General Random Turan}, which we restate below for convenience.

\begingroup
\def\thethm{\ref{Lemma:General Random Turan}}
\begin{prop}
Let $F$ be an $r$-graph with $r,|F|\ge 2$.  If there exist $C>0$ and positive functions $M=M(n)$ and $\tau=\tau(n,m)$ such that 
\begin{itemize}
    \item[(a)] $F$ is $(M,(\log n)^{C}, \tau)$-balanced, and
    \item[(b)] For all sufficiently large $n$ and $m\ge M(n)$, the function $\tau(n,m)$ is non-increasing with respect to $m$ and satisfies $\tau(n,m)\ge 1$, 
\end{itemize}
then there exists $C'\ge 0$ such that for all sufficiently large $n$, $m\ge M(n)$, and $0<p\le 1$ with $pm\rightarrow\infty$ as $n\rightarrow \infty$, we have a.a.s.
$$
\ex(G^r_{n,p},F)\le \max\l\{C'pm,\tau(n,m)(\log n)^{C'}\r\}.
$$
\end{prop}
\addtocounter{thm}{-1}
\endgroup

\begin{proof}
Let $n$ be sufficiently large and $m\ge M(n)$.  Note that by \Cref{prop:optimalBSBound} we have $m\ge \ex(n,F)=\Omega(n)$ since $r,|F|\ge 2$, and in particular for $n$ sufficiently large we have
\begin{equation}m\ge (\log n)^{C}.\label{eq:dumbMBound}\end{equation}
We will take $[n]=\{1,\dots,n\}$ to be the vertex set of all the $r$-graphs discussed below. We first prove the following claim by repeatedly applying Theorem~\ref{theorem:container}.

\begin{claim}\label{claim:container}
There exist a collection $\C$ of $r$-graphs on $[n]$ and a constant $\tilde{C}$ (independent of $n$ and $m$) such that
\begin{itemize}
    \item[(i)] Every $F$-free $r$-graph on $[n]$ is contained in some $S\in\C$;
    \item [(ii)] $|S|\le m$ for every $S\in\C$;
    \item [(iii)] $|\C|\le \exp\l((\log n)^{\tilde{C}}\tau(n,m)\r)$.
\end{itemize}
\end{claim}
\begin{proof}
We will iteratively define collections $\C_i$ satisfying (i) and the size condition $|\C_i|\le  \exp\l(\frac{r\log n\cdot\tau(n,m')}{\epsilon}\r)^i$ for some $\epsilon$ depending only on $F$ as follows.  We start with $\C_0=\{K^r_n\}$, where $K^r_n$ is the complete $r$-graph on $[n]$ and note that this collection satisfies (i) and $|\C_0|=1$. Given a collection $\C_{i-1}$ for  $i\ge 1$ which satisfies (i) and the size condition, if $\C_{i-1}$ satisfies (ii) then we stop and set $\C=\C_{i-1}$, and otherwise we construct a new collection $\C_i$ as following. Let
$$\C_{i-1}^{\le m}=\{G\in\C_{i-1}:|G|\le m\}~~\text{and}~~\C_{i-1}^{> m}=\{G\in\C_{i-1}:|G|>m\}.$$
For each $G\in \C_{i-1}^{>m}$, we construct a collection $\C_{G}$ of $r$-graphs on $[n]$ as follows.  Let $m'=|G|> m$.   By (a), there exists a collection $\c{H}$ of copies of $F$ in $G$ such that, $\forall 1\le i\le |F|$,
$$
\Delta_i(\c{H})\le \frac{(\log n)^{C}|\c{H}|}{m'}\l(\frac{\tau(n,m')}{m'}\r)^{i-1}.
$$
By applying Theorem~\ref{theorem:container} with $B=\tau(n,m')\ge 1$, and $L=m'(\log n)^{-C}\ge 1$ (here we use \eqref{eq:dumbMBound} and $m'>m$), we obtain a collection $\C_{G}$ of $r$-graphs on $[n]$ (i.e. subsets of ${[n]\choose r}$) such that
\begin{itemize}
    \item[(i')] every $F$-free subgraph of $G$ is contained in some $G'\in\C_{G}$;
    \item[(ii')] $|G'|\le \l(1-\epsilon(\log n)^{-C}\r)m'$ for every $G'\in\C_{G}$;
    \item [(iii')] $|\C_{G}|\le \exp\l(\frac{r\log n\cdot\tau(n,m')}{\epsilon}\r)$.
\end{itemize}
Here $\epsilon$ is a constant depending only on $|F|$. Having constructed $\C_{G}$ for each $G\in \C_{i-1}^{>m}$, we finally define
$$
\C_i=\C_{i-1}^{\le m}\cup\bigcup_{G\in\C_{i-1}^{>m}}\C_{G}.
$$
Note that since we inductively assumed (i) held for $\C_{i-1}$, property (i') implies that $\C_i$ also satisfies (i).  Similarly (iii') and the size condition on $\C_{i-1}$ implies $|\C_i|\le  \exp\l(\frac{r\log n\cdot\tau(n,m')}{\epsilon}\r)^i$ as desired.

Let $t$ be such that $\C=\C_t$ (we implicitly show such a $t$ exists below).  Note that (ii) holds by construction of $\C$, and (i) holds for $\C_t=\C$, so it so it remains to check (iii).  Observe that by (ii'), we have for every $G\in \C_i$ that 
$$
|G|\le \max\l\{m, \l(1-\epsilon(\log n)^{-C}\r)^i\binom{n}{r}\r\}.
$$
Note that when $i\ge (\log n)^{\tilde{C}'}$ for some sufficiently large constant $\tilde{C}'$, we have
$$
\l(1-\epsilon(\log n)^{-C}\r)^i\binom{n}{r}<m.
$$
Thus we have $t\le (\log n)^{\tilde{C}'}$, and hence by the size condition
\[|\C|=|\C_t|\le  \exp\l(\frac{r\log n\cdot\tau(n,m')}{\epsilon}\r)^{(\log n)^{\tilde{C}'}}\le  \exp\l((\log n)^{\tilde{C}}\tau(n,m)\r),\]
where $\tilde{C}$ is some suitably large constant, completing the proof of the claim.
\end{proof}

Let $E=\max\l\{C'pm,\tau(n,m)(\log n)^{C'}\r\}$ where $C'$ is a sufficiently large constant, and let $X_E$ denote the number of $F$-free $E$-edge subgraphs of $G^r_{n,p}$. By Claim~\ref{claim:container} and linearity of expectation, we have
$$
\E[X_E]\le |\C|\binom{m}{E}p^E\le \exp\l((\log n)^{C'''}\tau(n,m)+\log\l(\frac{emp}{E}\r)\cdot E\r)\le \exp(-E).
$$
Since $E\ge C_2pm\rightarrow\infty$ as $n\rightarrow\infty$, by Markov's inequality, we have
$\Pr[X_E\ge 1]\le \E[X_E]\rightarrow 0$ as $n\rightarrow\infty$. Therefore, a.a.s.
$$
\ex(G^r_{n,p},F)\le E=\max\l\{(\log n)^{C_2}\tau(n,m),~C_2pm\r\}.
$$
\end{proof}
Combining our results here gives a criteria to establish tight bounds for random Tur\'an problems.
\begin{cor}\label{cor:randomTuranBounds}
Let $F$ be an $r$-graph with $\Delta(F)\ge 2$ such that $F$ is $(M,\gamma,\tau)$-balanced with $M(n)=\Theta(\ex(n,F))$, $\gamma=(\log n)^{\Theta(1)}$ and $\tau=n^{r-\frac{1}{d_r(F)}}(\log n)^{\Theta(1)}$. Then
$$\ex(G^r_{n,p}, F)=
    \l\{
    \begin{aligned}
        & \max\{\Theta(p \cdot \ex(n,F)),n^{r-\frac{1}{d_r(F)}}(\log n)^{\Theta(1)}\},~~&\t{if}~p\gg n^{-\frac{1}{d_r(F)}}\\
        &(1+o(1))p\binom{n}{r},~~&\t{if}~ n^{-r}\ll p\ll  n^{-\frac{1}{d_r(F)}}.\\
    \end{aligned}
    \r. $$
\end{cor}
\begin{proof}
The upper bound comes from applying \Cref{Lemma:General Random Turan} with $m=M(n)$, and the lower bound comes from \Cref{lem:generalLowerRandomTuran}.
\end{proof}

We now prove \Cref{thm:optimalBalanced} which we recall below.

\begingroup
\def\thethm{\ref{thm:optimalBalanced}}
\begin{thm}
Let $F$ be an $r_0$-graph with $2\le \Delta(F)<|F|$ which is the spanning subgraph of a tight $r_0$-tree. 
 If there exists a constant $C_{r_0}$ such that $F$ is $(M_{r_0},\gamma_{r_0},\tau_{r_0})$-balanced where $M_{r_0}(n)=C_{r_0}n^{r_0-1}$, $\gamma_{r_0}(n)= (\log n)^{C_{r_0}}$ and $\tau_{r_0}(n,m)= n^{r_0-\frac{1}{d_{r_0}(F)}}(\log n)^{C_{r_0}}$; then for all $r\ge r_0$, we have
     $$
    \ex(G^r_{n,p}, F^{(r)})=
    \l\{
    \begin{aligned}
        & \max\{\Theta(pn^{r-1}),n^{r-\frac{1}{d_r(F^{(r)})}}(\log n)^{\Theta(1)}\},~~&\t{if}~p\gg n^{-\frac{1}{d_r(F^{(r)})}}\\
        &(1+o(1))p\binom{n}{r},~~&\t{if}~ n^{-r}\ll p\ll  n^{-\frac{1}{d_r(F^{(r)})}}.\\
    \end{aligned}
    \r.
    $$
\end{thm}
\addtocounter{thm}{-1}
\endgroup

\begin{proof}
    By \Cref{prop:optimalBSExpansions} (whose hypothesis is exactly the same as that of \Cref{thm:optimalBalanced}), we find for all $r\ge r_0$ that $F^{(r)}$ is $(M_r,\gamma_r,\tau_r)$-balanced with $M(n)=\Theta(n^{r-1})$, $\gamma=(\log n)^{\Theta(1)}$ and $\tau=n^{r-\frac{1}{d_r(F)}}(\log n)^{\Theta(1)}$.  
    
    We claim that $\ex(n,F^{(r)})=\Theta(n^{r-1})$.  Indeed, the lower bound comes from a star (i.e.\ every edge containing a given vertex $v\in [n]$), which is $F$-free since $\Delta(F)<|F|$ (meaning no vertex of $F^{(r)}$ is contained in every edge of $F^{(r)}$); and the upper bound comes from $M(n)=O(n^{r-1})$ and \Cref{prop:optimalBSExpansions}.  We conclude the result by \Cref{cor:randomTuranBounds}.
\end{proof}

\section{Proofs of Main Results}\label{sec:applying}
We now apply our Lifting Theorems to prove our results.  We begin by proving our main result \Cref{thm:strongExpansion} for large expansions in \Cref{sub:strongExpansion}, after which we give short proofs of Theorems~\ref{theorem:K2t_RT} and \ref{theorem:theta_RT} giving tight bounds for the random Tur\'an problem for expansions of $K_{2,t}$ and theta graphs $\theta_{a,b}$ for smaller values of $r$ in \Cref{sub:short}.  The details of the related result \Cref{thm:Kst} for $K_{s,t}^{(r)}$ with general $s$ is deferred to an arXiv-only appendix as the computations are cumbersome and we plan to prove even stronger bounds for $s\ge 3$ in forthcoming work \cite{nie202X}.

\subsection{Proof of \Cref{thm:strongExpansion}}\label{sub:strongExpansion}

The statement of \Cref{thm:strongExpansion} involves Sidorenko hypergraphs, for which we make a related definition.
\begin{defn}
    We say that a hypergraph is $(C,\al)$-Sidorenko if every $n$-vertex $r$-graph $H$ with $|H|\ge C n^{\al}$ contains at least $C^{-1} (|H| n^{-r})^{|F|}n^{v(F)}$ copies of $F$.
\end{defn}
It is well known that being $(C,\al)$-Sidorenko for some $\al<r$ is equivalent to being Sidorenko, though for our purposes, we will only need the following easy-to-prove direction of this relationship.
\begin{lem}\label{lem:Sidorenko_to_Supersaturation}
    If $F$ is a non-empty $r$-graph which is Sidorenko, then there exists some $C>0$ such that $F$ is $(C,r-\frac{1}{|F|})$-Sidorenko.
\end{lem}
\begin{proof}
    Let $C$ be a sufficiently large constant and let $H$ be an $r$-graph with $|H|\ge Cn^{r-1/|F|}$.  Because $F$ is Sidorenko, the number of homomorphisms from $F$ to $H$ is at least $(|H|n^{-r})^{|F|}n^{v(F)}\ge C^{|F|}n^{v(F)-1}$. 
 Since the number of non-injective homomorphisms from $F$ to $H$ is trivially at most $O(n^{v(F)-1})$, we conclude that, given $C$ is sufficiently large, at least a half of the homomorphisms from $F$ to $H$ are injective, i.e.\ there are at least $\half (|H|n^{-r})^{|F|}n^{v(F)}$ genuine copies of $F$ in $H$. 
 Thus taking $C$ sufficiently large gives the result.
\end{proof}

We next state a hypergraph analog of a balanced supersaturation result for general graphs due to Jiang and Longbrake~\cite[Theorem 3.7]{jiang2022balanced}.  For this, given an $r$-graph $F$ with $|F|\ge 3$, we define its \emph{proper $r$-density} $d^*_r(F)$ by \[d^*_r(F):=\max_{F'\subsetneq F,\ |F'|\ge 2}\frac{|F'|-1}{v(F')-r}.\]

\begin{thm} \label{thm:balanced-supersat_r}
Let $F$ be a non-empty $r$-partite $r$-graph with $|F|\ge 3$. If there exist positive reals $C$ and $\al>r-\frac{1}{d_r(F)}$ such that $F$ is $(C,\al)$-Sidorenko, then there exists a constant $C'=C'(F)>0$ such that $F$ is $(M,\gamma,\tau)$-balanced where $M(n)=C'n^{r-\frac{1}{|F|}}$, $\gamma(n)=C'$, and 
$$
\tau(n,m)=C'\max\l\{n^{r-\frac{v(F)-r}{|F|-1}},~m^{1-\frac{1}{d_r^*(F)(r-\al)}}n^{\frac{\al}{d^*_r(F)(r-\al)}}\r\}.
$$
\end{thm}
The proof of \Cref{thm:balanced-supersat_r} is an entirely straightforward generalization of \cite[Theorem 3.7]{jiang2022balanced} from graphs to hypergraphs, and as such we defer the exact details of this argument to an arXiv-only appendix. 

The following lemma provides some inequalities that will be useful in the upcoming proofs.  For this, we recall from the introduction that we say an $r$-graph $F$ with $|F|\ge 3$ is strictly $r$-balanced if $d_r(F)=\frac{|F|-1}{v(F)-r}$ and if (in terms of the new notation above) we have $d_r(F)>d^*_r(F)$.

\begin{lem}\label{lemma:inequalities}
    If $F$ is a strictly balanced $r_0$-graph with  $\Del(F)\ge 2$ and $r\ge |F|^2v(F)r_0$, then 
    \begin{itemize}
        \item[(i)] $$r_0-\frac{v(F)-r_0}{|F|-1}>0,$$
        \item[(ii)] $$\frac{r_0|F|-1}{d^*_{r_0}(F)}+\l(1-\frac{|F|}{d_{r_0}^*(F)}\r)\l(\frac{r-r_0}{r-1}\r)>0,$$
        \item[(iii)] $$\l(r_0-\frac{1}{|F|}\r)\frac{r-1}{r_0-1}-\frac{r-r_0}{r_0-1}<r-1,$$
        \item[(iv)] $$r_0-\frac{v(F)-r_0}{|F|-1}\ge \frac{r_0|F|-1}{d_{r_0}^*(F)}+\l(r_0-\frac{r_0-1}{r-1}\r)\l(1-\frac{|F|}{d_{r_0}^*(F)}\r).$$
    \end{itemize}
\end{lem}

\begin{proof}
    Part \textit{(i)} is equivalent to saying $v(F)<r_0 |F|$. And indeed, since $F$ is balanced and has at least one edge, $F$ has no isolated vertices, implying that $v(F)\le r_0|F|$ with equality only if $F$ is a matching.  However, $F$ can not be a matching since $\Del(F)\ge 2$, so this inequality must be strict.

    For part \textit{(ii)}, since $0<\frac{r-r_0}{r-1}<1$, the lefthand side of the inequality is larger than
$$
\frac{r_0|F|-1}{d^*_{r_0}(F)}-\frac{|F|}{d^*_{r_0}(F)}\ge \frac{|F|-1}{d^*_{r_0}(F)}>0.
$$

    For part \textit{(iii)}, moving all monomials with $r$ to the righthand side gives
    $$
    \frac{1}{|F|(r_0-1)}+1<\frac{r}{|F|(r_0-1)},
    $$
    which is equivalent to $r> |F|(r_0-1)+1$. This clearly holds since $r\ge |F|^2v(F)r_0$.

    For part \textit{(iv)}, by moving all of the terms in the inequality depending on $r$ to one side and using that $\frac{v(F)-r_0}{|F|-1}=\frac{1}{d_{r_0}(F)}$ since $F$ is balanced, we find that the stated inequality is equivalent to
    \begin{equation*}
    \frac{1}{d_{r_0}^*(F)}-\frac{1}{d_{r_0}(F)}\ge \frac{r_0-1}{r-1}\left(\frac{|F|}{d^*_{r_0}(F)}-1\right),
    \end{equation*}
    which after multiplying by $d_{r_0}(F)d_{r_0}(F)$ on both sides gives
    \[d_{r_0}(F)-d_{r_0}^*(F)\ge \frac{r_0-1}{r-1}(|F|-d^*_{r_0}(F))d_{r_0}(F).\]

    We claim that \[
    d_{r_0}(F)-d^*_{r_0}(F)\ge \frac{1}{\l(v(F)-r_0\r)^2}.
    \]  
    Indeed, this follows from the elementary fact that if $a/b$ and $c/d$ are rationals with $a/b>c/d$ and $1\le b,d\le t$, then
    \[
    \frac{a}{b}-\frac{c}{d}=\frac{ad-bc}{bd}\ge \frac{1}{t^2}.
    \]

    On the other hand,
    $$
    \frac{r_0-1}{r-1}\cdot \l(|F|-d^*_{r_0}(F)\r)\cdot d_{r_0}(F)\le \frac{r_0}{r}\cdot |F|\cdot \frac{|F|}{v(F)-r_0}.
    $$
    Thus it suffices to check that
    $$
    \frac{1}{\l(v(F)-r_0\r)^2}\ge \frac{r_0}{r}\cdot |F|\cdot \frac{|F|}{v(F)-r_0},
    $$
    which is true since $r\ge |F|^2v(F)r_0$.

\end{proof}

Using the Lifting Theorems (\Cref{lemma:BSviaSHADOW} and \Cref{lemma:BSviaGreedy}) together with \Cref{lem:Sidorenko_to_Supersaturation} and \Cref{thm:balanced-supersat_r}, we are now ready to show that $F^{(r)}$ has optimal balanced supersaturation when $r$ is sufficiently large.

\begin{thm}\label{thm:expansion_optimal_balanced}
Let $F$ be a strictly $r_0$-balanced $r_0$-graph with $\Delta(F)\ge 2$. If $F$ is Sidorenko, then for all $r> |F|^2v(F)r_0$, there exists a constant $C>0$ such that $F^{(r)}$ is $(M,\gamma,\tau)$-balanced where $M(n)=Cn^{r-1}$, $\gamma=(\log n)^C$, and 
$$
\tau(n,m)=n^{r-\frac{1}{d_r(F^{(r)})}}(\log n)^C.
$$
\end{thm}
\begin{proof}
Since $F$ is Sidorenko, we know that $F$ is $r_0$-partite, and by \Cref{lem:Sidorenko_to_Supersaturation}, we know that there exists a constant $C_1$ such that $F$ is $(C_1,r_0-\frac{1}{|F|})$-Sidorenko. Note that $d_{r_0}(F)=\frac{|F|-1}{v(F)-r_0}< |F|$, which implies $r_0-\frac{1}{|F|}> r_0-\frac{1}{d_{r_0}(F)}$. Thus by \Cref{thm:balanced-supersat_r}, there exists a constant $C_{r_0}>0$ such that $F$ is $(M_{r_0}, \gamma_{r_0}, \tau_{r_0})$-balanced where $M_{r_0}(n)=C_{r_0}n^{r_0-\frac{1}{|F|}}$, $\gamma_{r_0}(n)=C_{r_0}$, and 
$$
\tau_{r_0}(n,m)=C_{r_0}\max\l\{n^{r_0-\frac{v(F)-r_0}{|F|-1}},~m^{1-\frac{|F|}{d_{r_0}^*(F)}}n^{\frac{r_0|F|-1}{d^*_{r_0}(F)}}\r\}.
$$
Let $r_1=|F|^2v(F)r_0$. We want to apply \Cref{lemma:BSviaSHADOW} to ``lift'' the balanced supersaturation results for $F$ to $F^{(r_1)}$. To this end, we need to check that conditions (a), (b), (c) and (d) in \Cref{lemma:BSviaSHADOW} hold. Conditions (a), (b) and (c) (that $F$ is $(M_{r_0},\gam_{r_0},\tau_{r_0})$-balanced, that $M_{r_0}(n)n^{-\frac{r_1-r_0}{r_1-1}}$ is non-decreasing, and that $\gam_{r_0}(n)$ is non-decreasing) are straightforward to verify. To check condition (d) (that $\tau_{r_0}(n,m)$ is non-increasing with $m$ and that $\tau_{r_0}(nx,mx^{\frac{r_1-r_0}{r_1-1}})$ is non-decreasing with $x$), note that the exponent of $m$ in $\tau_{r_0}(n,m)$ is either 0 or $1-\frac{|F|}{d^*_{r_0}(F)}<0$, and that the exponent of $x$ in $\tau_{r_0}(nx,mx^{\frac{r_1-r_0}{r_1-1}})$ is either
$$
r_0-\frac{v(F)-r_0}{|F|-1}\text{~~or~~}\frac{r_0|F|-1}{d^*_{r_0}(F)}+\l(1-\frac{|F|}{d_{r_0}^*(F)}\r)\l(\frac{r_1-r_0}{r_1-1}\r),
$$
which are both positive by \Cref{lemma:inequalities}. Thus condition (d) holds. By \Cref{lemma:BSviaSHADOW} we know that there exists a sufficiently large constant $C_{r_1}$ such that $F^{(r_1)}$ is $(M_{r_1}, \gamma_{r_1}, \tau_{r_1})$-balanced where
$$
M_{r_1}(n)=\max\l\{M_{r_0}(n)^{\frac{r_1-1}{r_0-1}}n^{-\frac{r_1-r_0}{r_0-1}},~n^{r_1-1}\r\}(\log n)^{C_{r_1}}=n^{r_1-1}(\log n)^{C_{r_1}},
$$
where here the second equality used \Cref{lemma:inequalities}(iii),

$$
\gamma_{r_1}(n)=\gamma_{r_0}(n)\l(\log n\r)^{C'}\le \l(\log n\r)^{C_{r_1}},
$$
and
$$
\begin{aligned}
\tau_{r_1}(n,m)&=\tau_{r_0}\l(n,n^{\frac{r_1-r_0}{r_1-1}}m^{\frac{r_0-1}{r_1-1}}(\log n)^{-C'}\r)(\log n)^{C'}\\
&\le \max\l\{n^{r_0-\frac{v(F)-r_0}{|F|-1}},~\l(n^{\frac{r_1-r_0}{r_1-1}}m^{\frac
{r_0-1}{r_1-1}}\r)^{\l(1-\frac{|F|}{d_{r_0}^*(F)}\r)}n^{\frac{r_0|F|-1}{d^*_{r_0}(F)}}\r\}(\log n)^{C_{r_1}}\\
&\le \max\l\{n^{r_0-\frac{v(F)-r_0}{|F|-1}},~n^{\l(\frac{r_1-r_0}{r_1-1}+r_0-1\r)\l(1-\frac{|F|}{d_{r_0}^*(F)}\r)}n^{\frac{r_0|F|-1}{d^*_{r_0}(F)}}\r\}(\log n)^{C_{r_1}}\\
&=n^{r_0-\frac{v(F)-r_0}{|F|-1}}(\log n)^{C_{r_1}},
\end{aligned}
$$
where here the second inequality uses the fact that $m\ge M_{r_1}(n)\ge  n^{r_1-1}$ and that $1-\frac{|F|}{d^*_{r_0}(F)}<0$, and the last equality is implied by \Cref{lemma:inequalities}(iv).

For $r>r_1$, we want to get rid of the logarithmic factor in $M(n)$ by using \Cref{lemma:BSviaGreedy} on $F^{(r-1)}$. Note that $r>r_1>v(F)$, and thus $F^{(r)}$ is a spanning subgraph of a tight $r$-tree by \Cref{lem:spanningTreeInduction} (since the $r_0$-shadow of a single $v(F)$-edge contains $F$ as a spanning subgraph). Suppose that we have inductively proven that there exists a constant $C_{r-1}$ such that $F^{(r-1)}$ is $(M_{r-1},\gamma_{r-1},\tau_{r-1})$-balanced where $M_{r-1}(n)=n^{r-2}(\log n)^{C_{r-1}}$, $\gamma_{r-1}(n)=(\log n)^{C_{r-1}}$, and 
$$
\tau_{r-1}(n,m)= n^{r_0-\frac{v(F)-r_0}{|F|-1}}(\log n)^{C_{r-1}},
$$
noting that we have proved this for the case $r-1=r_1$. It is straightforward to check that conditions (a), (b) and (c) in \Cref{lemma:BSviaGreedy} hold. Let $C_r$ be a sufficiently large constant and let $M_r(n)=C_rn^{r-1}$, $\gamma_r(n)=\gamma_{r-1}(n)\cdot C_r(\log n)$, and let
$$
A=A(n,m)=\frac{m}{M_{r-1}(n)\cdot C_r\log n}=\frac{m}{n^{r-2}\cdot C_r(\log n)^{C_{r-1}+1}}.
$$
Clearly, condition (d) in \Cref{lemma:BSviaGreedy} holds. Thus by \Cref{lemma:BSviaGreedy}, $F^{(r)}$ is $(M_r, \gamma_r,\tau_r)$-balanced where there exists a sufficiently large constant $C$ such that
$$
\tau_r(n,m)=\max\l\{A(n,m)^{-\frac{1}{d_r(F^{(r)})}}\cdot m,~\tau_{r-1}\l(n,~\frac{m}{A(n,m)\cdot C_{r}\log n}\r)\cdot C_r\log n\r\}\le n^{r-\frac{1}{d_r(F^{(r)})}}(\log n)^C,
$$
where here the last inequality comes from the definition of $\tau_{r-1}$ together with the fact from \Cref{lem:densityRelation} that
$$
r_0-\frac{v(F)-r_0}{|F|-1}=r_0-\frac{1}{d_{r_0}(F)}=r-\frac{1}{d_r(F^{(r)})},
$$
and that (using $m\ge n^{r-1}$)
$$
A(n,m)^{-\frac{1}{d_r(F^{(r)})}}\cdot m= m^{1-\frac{1}{d_r(F^{(r)})}}n^{\frac{r-2}{d_r(F^{(r)})}+o(1)}\le n^{r-1-\frac{1}{d_r(F^{(r)})}+o(1)}.
$$
This completes the proof.  
\end{proof}
As an aside, we note that it is perhaps a little surprising that although the general result of Jiang and Longbrake \Cref{thm:balanced-supersat_r} typically does not give optimal balanced supersaturation results for most $r_0$-graphs $F$, the proof of \Cref{thm:expansion_optimal_balanced} shows that our Lifting Theorems are powerful enough to translate these potentially suboptimal bounds into optimal bounds for high enough expansions.

Finally, we will need the following result from \cite[Theorem 1.4]{nie2023sidorenko} in order to prove the ``moreover'' portion of \Cref{thm:strongExpansion}.

\begin{thm}[\cite{nie2023sidorenko}]\label{thm:Sidorenko}
    If $F$ is an $r$-graph with $|F|\ge 2$ and $\frac{v(F)-r}{|F|-1}<r$ which is not Sidorenko,  then there exists some $\ep>0$ such that for any $p=p(n)\ge n^{-\frac{v(F)-r}{|F|-1}}$, we have a.a.s.
    \[\ex(G_{n,p}^r,F)\ge n^{r-\frac{v(F)-r}{|F|-1}-o(1)}(p n^{\frac{v(F)-r}{|F|-1}})^{\ep}.\]
\end{thm}
The exact statement of \cite[Theorem 1.4]{nie2023sidorenko} in fact replaces the $\ep$ here with a more precise quantity $\frac{s(F)}{e(F)-1+s(F)}$ where $s(F)$ measures how ``far'' from Sidorenko $F$ is, but all that we need here is $F$ being non-Sidorenko implies $s(F)>0$.  With this, we now have all the pieces needed to prove \Cref{thm:strongExpansion}.

\begin{proof}[Proof of \Cref{thm:strongExpansion}]
    The first part of the theorem regarding Sidorenko $F$ follows immediately from \Cref{thm:optimalBalanced} and \Cref{thm:expansion_optimal_balanced}.  For the moreover portion, let $F$ be as in the hypothesis of the theorem such that $F^{(r)}$ is not Sidorenko for all $r\ge r_0$.  By the definition of the expansion $F^{(r)}$ and $F$ being balanced, we have for all $r\ge r_0$ that
    \[r-\frac{v(F^{(r))}-r}{|F^{(r)}|-1}=r-\frac{v(F)+(r-r_0)|F|-r}{|F|-1}=r_0-\frac{v(F)-r_0}{|F|-1}=r_0-\frac{1}{d_{r_0}(F)},\]
    and by \Cref{lemma:inequalities}(i) this quantity is positive.  This together with \Cref{thm:Sidorenko} implies that for all $r\ge r_0$, there exists some $\ep>0$ such that for all $\del>0$ and $p=n^{-r+r_0-\frac{1}{d_{r_0}(F)}+\del}=n^{-r+\frac{v(F^{(r)})-r}{|F^{(r)}|-1}+\del}$, we have a.a.s.
    \[\ex(G_{n,p}^r,F^{(r)})\ge n^{r_0-\frac{1}{d_{r_0}(F)}+\ep\del-o(1)}.\]
    This bound is strictly stronger than what is claimed in \Cref{thm:strongExpansion} for Sidorenko hypergraphs whenever $0<\del<1$, proving that such bounds can not extend to any $F^{(r)}$.
\end{proof}

Before moving on, we emphasize that the bound of $r>|F|^2 v(F) r_0$ obtained in \Cref{thm:strongExpansion} can be improved upon significantly provided more information about the hypergraph $F$ is known.  Notably, one can obtain better bounds here provided one has better bounds on either (a) the difference $d_{r_0}(F)-d_{r_0}^*(F)$, (b) the smallest $r_1$ such that $F^{(r_1+1)}$ is contained in a spanning subgraph of a tight $(r_1+1)$-tree, or (c) the smallest $\al$ such that $F$ is $(C,\al)$-Sidorenko for some $C$.  In particular, our current proof uses the most pessimistic estimates possible for each of these quantities, so there is quite a bit of room for substantial improvement for specific choices of $F$.  

More precisely, if one uses more precise bounds on $d_{r_0}(F)-d^*_{r_0}(F)$, $r_1$ and $\alpha$, then a more involved calculation shows that the tight bounds of \Cref{thm:strongExpansion} would hold provided

$$
r\ge r_0-\frac{1}{d_{r_0}(F)}+\frac{(r_0-1)\l((r_1-r_0)d_{r_0}(F)+1\r)(\frac{1}{r_0-\alpha}-d^*_{r_0}(F))}{(r_1-1)(d_{r_0}(F)-d^*_{r_0}(F))}.
$$

In particular, when $F=C_{2\ell}$, we can use $r_1=r_0=2$, $d_2(C_{2\ell})=\frac{2\ell-1}{2\ell-2}$, $d^*_2(C_{2\ell})=1$, and $\alpha=1+\frac{1}{\ell}$ to obtain tight bounds for $r\ge 4$, recovering the known bounds for even cycles as stated in \Cref{thm:looseCycles}. However, we emphasize that one can often obtain better bounds for $r$ than even this optimized version of our result whenever more precise balanced supersaturation results are known beyond the general results of Jiang and Longbrake.  In particular, for the case $F=K_{2,t}$, this optimized result above gives tight bounds only for $r\ge t+2$, while (as we show below) the true answer is $r\ge 3$ for $t\ge 3$.

\subsection{Proofs of \Cref{theorem:K2t_RT} and \Cref{theorem:theta_RT}}\label{sub:short}
Our proof for theta graphs relies on \Cref{lemma:BSviaGreedy}, for which we need the following.

\begin{lem}\label{lemma:tighttree_Theta}
For all $a,b\ge 2$, the 3-graph $\theta^{(3)}_{a,b}$ is a spanning subgraph of a tight $3$-tree.
\end{lem}
\begin{proof}
By \Cref{lem:spanningTreeInduction} it suffices to construct a tight $3$-tree whose 2-shadow contains a copy of $\theta_{a,b}$. Let $u$ and $v$ be the two end points of $\theta_{a,b}$, let $P_i$, $1\le i\le a$ be the $a$ paths, and let $w_{i,j}$, $1\le j\le b-1$, be the internal vertices of $P_i$, and for notational convenience we let $w_{i,0}=u$ and $w_{i,b}=v$ for all $i$.  To build our tight tree, we begin by adding all of the edges of the form $\{w_{i,0},w_{i,1},w_{i,b}\}$ for all $i$, and iteratively given then we have added all edges of the form $\{w_{i,j-1},w_{i,j},w_{i,b}\}$ for some $1\le j\le b-2$ we add all edges of the form $\{w_{i,j},w_{i,j+1},w_{i,b}\}$.  It is not difficult to see that this is a tight tree (since each edge we add adds the new vertex $w_{i,j+1}$) and that its 2-shadow contains the edge $w_{i,j},w_{i,j+1}$ for all $1\le i\le a$ and $1\le j\le b-1$ (with the $j=b-1$ case coming from the final edges $\{w_{i,b-2},w_{i,b-1},w_{i,b}\}$, and hence a copy of $\theta_{a,b}$.
\end{proof}
We now show that $\theta^{(r)}_{a,b}$ has optimal balanced supersaturation whenever $\theta_{a,b}$ does with $a\ge 3$.
\begin{prop}\label{prop:thetaBS}
    Let $a\ge 3$ and $b\ge 2$ be such that there exists a constant $C_2>0$ such that $\theta_{a,b}$ is $(M_2,C_2,\tau_2)$-balanced where
    $$
    M_2(n)=C_2n^{1+\frac{1}{b}}
    $$
    and
    $$
    \tau_2(n,m)=\max\l\{n^{1+\frac{a-1}{ab-1}},~m^{-\frac{1}{b-1}}n^{1+\frac{2}{b-1}}\r\}.
    $$
    Then for all $r\ge 3$, we have for all $0<p=p(n)\le 1$ that a.a.s.

 $$
    \ex(G^r_{n,p}, \theta_{a,b}^{(r)})=
    \l\{
    \begin{aligned}
        & \max\{\Theta(pn^{r-1}),n^{1+\frac{a-1}{ab-1}}(\log n)^{\Theta(1)}\},~~&\t{if}~p\gg n^{-r+1+\frac{a-1}{ab-1}}\\
        &(1+o(1))p\binom{n}{r},~~&\t{if}~ n^{-r}\ll p\ll  n^{-r+1+\frac{a-1}{ab-1}}.\\
    \end{aligned}
    \r.
    $$
\end{prop}
\begin{proof}
We begin by using \Cref{lemma:BSviaGreedy} to show that $\theta_{a,b}^{(3)}$ is balanced, and for this we need to check the hypothesis of \Cref{lemma:BSviaGreedy}. By Lemma~\ref{lemma:tighttree_Theta}, $\theta^{(3)}_{a,b}$ is a spanning subgraph of a tight 3-tree. Property (a) holds by hypothesis, and it is easy to check that $\gamma_2$ and $\tau_2$ satisfy properties (b) and (c). Thus Theorem~\ref{lemma:BSviaGreedy} guarantees the existence of some sufficiently large constant $C_3$ such that the following holds. Let $M_3(n)=C_3n^2$, let $\gamma_3(n)=C_2C_3\log n$, and let 
$$
A=A(n,m)=m^{\frac{ab-1}{2ab-a-1}}n^{-\frac{ab+a-2}{2ab-a-1}}.
$$
It is easy to check that property (d) in Theorem~\ref{lemma:BSviaGreedy} holds, that is, for all sufficiently large $n$ and $m\ge M_3(n)=C_3n^{2}$ we have 
$$
m n^{-2} \le A(n,m)\le \frac{m}{M_2(n) C_3\log n}=\frac{m}{C_2C_3n^{1+\frac{1}{b}}\log n}.
$$
Define the function
$$
\tau_3(n,m)=\max\l\{A(n,m)^{-2+\frac{a-1}{ab-1}}m,~\tau_2\l(n,\frac{m}{A(n,m)\cdot C_3\log n}\r)\cdot C_3\log n\r\}.
$$
By Theorem~\ref{lemma:BSviaGreedy}, we have that $\theta_{a,b}^{(3)}$ is $(M_3, \gamma_3, \tau_3)$-balanced.  Note that
$$A(n,m)^{-2+\frac{a-1}{ab-1}}m=n^{1+\frac{a-1}{ab-1}},$$
and we claim that
$$
\tau_2\l(n,\frac{m}{A(n,m)}\r)=\max\l\{n^{1+\frac{a-1}{ab-1}},~m^{-\frac{a}{2ab-a-1}}n^{1+\frac{3a}{2ab-a-1}}\r\}\le n^{1+\frac{a-1}{ab-1}}.
$$
Indeed, using $m\le C_3 n^2$ and that $n$ is sufficiently large, it suffices to verify that $\frac{a-1}{ab-1}>\frac{a}{2ab-a-1}$, for which it suffices to show (using $a\ge 3$)
\begin{align*}
    (a-1)(2ab-a-1)-(ab-1)a&=(b-1)a^2-2ab+a+1\\ &\ge (b-1)3a-2ab+a+1=(b-2)a+1>0
\end{align*}

which holds for $b\ge 2$.  Thus, we have 
$$
\tau_3(n,m)\le n^{1+\frac{a-1}{ab-1}}(\log n)^{C_3}.
$$

Letting $\tau_3'(n,m)=n n^{1+\frac{a-1}{ab-1}}(\log n)^{C_3}$, the above shows that $\theta^{(3)}_{a,b}$ is $(M_3,\gamma_3,\tau_3')$-balanced and a spanning subgraph of a tight $3$-tree.  The desired random Tur\'an bound then follows for all $r\ge 3$ by \Cref{thm:optimalBalanced}.
\end{proof}
To utilize \Cref{prop:thetaBS}, we make use of the following known balanced supersaturation results (rephrased in the language of our paper).

\begin{thm}[ \cite {morris2016number} Theorem 7.4]\label{theorem:MorrisSaxtonSupersat}
For all $t\ge s\ge 2$, there exist sufficiently large $C_2>0$  such that $K_{s,t}$ is $(M_2, C_2, \tau_2)$-balanced where
$$
M_2(n)=C_2n^{2-1/s}
$$
and
$$
\tau_2(n,m)=\max\l\{n^{\frac{2st-s-t}{st-1}},~m^{-s+1}n^{2s-1}\r\}.
$$
\end{thm}

\begin{thm}[\cite{mckinley2023random} Corollary 5.2]\label{theorem:MS_BS_Theta}
For all $a\ge 100$ and $b\ge 3$, there exists a sufficiently large constant $C_2$ such that $\theta_{a,b}$ is $(M_2,C_2,\tau_2)$-balanced where
$$
M_2(n)=C_2n^{1+\frac{1}{b}}
$$
and
$$
\tau_2(n,m)=\max\l\{n^{1+\frac{a-1}{ab-1}},~m^{-\frac{1}{b-1}}n^{1+\frac{2}{b-1}}\r\}.
$$
\end{thm}

Theorems~\ref{theorem:K2t_RT} and \ref{theorem:theta_RT} now follow immediately from \Cref{prop:thetaBS} together with these balanced supersaturation results (and the fact that $K_{2,t}=\theta_{t,2}$).




\bibliographystyle{abbrv}
\bibliography{refs}

\newpage 
\appendix

\section{Appendix}

\subsection{Comparing the Methods}\label{sub:comparison}

Here we very informally show that \Cref{lemma:BSviaGreedy}, whose proof used both Greedy and Shadow Expansion, is essentially always at least as strong as \Cref{lemma:BSviaSHADOW}, whose proof only used Shadow Expansion.  For this we use the notation $O^*$ and $\Omega^*$ to denote asymptotic upper bounds and lower bounds that hold up to polylogarithmic factors, and throughout our argument we will be loose with logarithmic terms in order to make some of the calculations clearer.

Let $F$ be an $(r-1)$-graph which satisfies the hypothesis of \Cref{lemma:BSviaGreedy} by being $(M_{r-1},\gamma_{r-1},\tau_{r-1})$-balanced, and let us make the extra mild assumption that $\gamma_{r-1}(n)=O^*(1)$.  As we show in \Cref{prop:optimalBSBound}, this implies $\tau_{r-1}(n,m)=\Om^*(n^{r-\frac{1}{d_r(F^{(r)})}})$.  With this in mind, we consider applying \Cref{lemma:BSviaGreedy} with $A=(m/n)^{1/(r-1)}$ for $m\ge \max\{M_{r-1}^{\frac{r-1}{r-2}}(n)n^{-\frac{1}{r-2}},n^{r-1}\}$ (i.e.\ for the range of $m$ that \Cref{lemma:BSviaSHADOW} is valid for).  Note that $m\ge M_{r-1}^{\frac{r-1}{r-2}}(n)n^{-\frac{1}{r-2}}$ is equivalent to having $A=(m/n)^{1/(r-1)}\le m/M_{r-1}(n)$, so (up to logarithmic factors) we have that the upper bound of property (d) in \Cref{lemma:BSviaGreedy}, and it is straightforward to verify the lower bound $A\ge m n^{1-r}$ for $r\ge 3$ and $m\ge n^{r-1}$, so we can indeed apply \Cref{lemma:BSviaGreedy} with this value of $A$.  We now observe that the first term in the maximum of $\tau_r(n,m)$ defined in \Cref{lemma:BSviaGreedy} satisfies
\[A^{-\frac{1}{d_r(F^{(r))}}}m\le n^{r-\frac{1}{d_r(F^{(r)})}},\]
with the inequality using $m\le n^r$.  This together with the claim above implies that up to logarithmic terms $\tau_r(n,m)$ equals $\tau_{r-1}(n,A^{-1}m)=\tau_{r-1}(n,n^{\frac{1}{r-1}}m^{\frac{r-2}{r-1}})$, and this is exactly what \Cref{lemma:BSviaSHADOW} gives.

\subsection{Proof of \Cref{thm:balanced-supersat_r}}
We make use of the following lemma of Jiang and Longbrake~\cite{jiang2022balanced}.
\begin{lem}[Lemma 3.4,~\cite{jiang2022balanced}]\label{lemma:JL_key}
Let $r,h,\ell$ be positive integers, where $h\geq r\geq 2$.
Let $F$ be an $r$-partite $r$-graph with $h$ vertices and $\ell$ edges. Let $\alpha, C$ be positive reals such that for each $n$ every 
$n$-vertex $r$-graph $H$ with $m\geq Cn^\alpha$ edges contains at least $f(n,m)$ copies of $H$,
where $f$ is a function satisfying the following.
\begin{itemize}
\item[(a)] There is a constant $\delta>0$ such that for all $p\in (0,1]$ and  all positive reals $n,m$ 
\[ f(n,m)\geq \delta m^{\frac{h-\al}{r-\al}} / n^{\frac{\al(h-r)}{r-\al}} \quad \mbox{ and  }  \quad f(np,mp^r)\geq \delta f(n,m) p^h.\]
\item[(b)] For fixed $n$, $f(n,m)$ is increasing and convex in $m$.
\end{itemize}
Then there exist a constant $C'=C'(F) \geq 1$ such that if $H$ is an $n$-vertex $r$-graph with $kn^\alpha$ edges, where $k\geq C'$, then there exists
a family $\F$ of copies of $F$ in $H$ satisfying
\begin{itemize}
\item[(c)] $|\F|\geq \delta f(n,\frac{1}{8}|H|)$, and
\item[(d)] $\forall 1\leq i\leq \ell -  1$, \[\Delta_i(\F)\leq C' k^{-\frac{i - 1}{d^*_r(F)(r-\al)}} \frac{|\F|}{|H|},\]
where $d^*_r(F)$ is the proper $r$-density of $F$.
\end{itemize}
\end{lem}

We are now ready to prove \Cref{thm:balanced-supersat_r}, which we restate below for convenience.

\begingroup
\def\thethm{\ref{thm:balanced-supersat_r}}

\begin{thm} 
Let $F$ be a non-empty $r$-partite $r$-graph with $|F|\ge 2$. If there exist positive reals $C$ and $\al>r-\frac{1}{d_r(F)}$ such that $F$ is $(C,\al)$-Sidorenko, then there exists constant $C'=C'(F)>0$ such that $F$ is $(M,\gamma,\tau)$-balanced where $M(n)=C'n^{r-\frac{1}{|F|}}$, $\gamma=C'$, and 
$$
\tau(n,m)=C'\max\l\{n^{r-\frac{v(F)-r}{|F|-1}},~m^{1-\frac{1}{d_r^*(F)(r-\al)}}n^{\frac{\al}{d^*_r(F)(r-\al)}}\r\}.
$$

\end{thm}
\addtocounter{thm}{-1}
\endgroup

\begin{proof}
Since $F$ is $(C,\al)$-Sidorenko, for every $n$-vertex $r$-graph $H$ with $m\ge Cn^{\al}$ edges contains at least $C^{-1}m^{|F|}n^{-r|F|+v(F)}$ copies of $F$. We want to apply \Cref{lemma:JL_key} with
$$f(n,m)=C^{-1}m^{|F|}n^{-r|F|+v(F)}.$$ To this end, we need to check that conditions (a) and (b) in \Cref{lemma:JL_key} are satisfied. Since $|F|\ge 2$, we know that $v(F)\ge r+1$, and hence $m^{\frac{v(F)-\al}{r-\al}}/n^{\frac{\al(v(F)-r)}{r-\al}}$ is maximized when $\al$ is minimized. Thus by $\al> r-\frac{1}{d_r(F)}$ and $m\le n^r$ we have
$$
\frac{f(n,m)n^{\frac{\al(v(F)-r)}{r-\al}}}{m^{\frac{v(F)-\al}{r-\al}}}> C^{-1}m^{|F|-1-d_r(v(F)-r)}n^{-r|F|+v(F)+(d_rr-1)(v(F)-r)}=C^{-1}\l(\frac{n^r}{m}\r)^{d_r(v(F)-r)-|F|+1} \ge C^{-1}.
$$
Also note that
$$
\frac{f(np,mp^r)}{f(n,m)p^{v(F)}}=1.
$$
Thus condition (a) in \Cref{lemma:JL_key} holds with $\delta=C^{-1}$. Further, note that the exponent of $m$ in $f(n,m)$ is $|F|\ge 2$, thus condition (b) holds. By \Cref{lemma:JL_key}, there exists a constant $C_1\ge 1$ such that if $H$ is an $n$-vertex $r$-graph with $m\ge C_1n^{\al}$ edges, then there exists a family $\F$ of copies of $F$ in $H$ such that
\begin{equation}\label{equation:supersaturation}
|\F|\ge f(n,\frac{1}{8}m)= C^{-1}8^{-|F|}m^{|F|}n^{-r|F|+v(F)},
\end{equation}
and, $\forall 1\le i \le |F|-1$,
$$
\Delta_{i}(\F)\le C_1 \l({m}{n^{-\al}}\r)^{-\frac{i-1}{d^*_r(F)(r-\al)}}\frac{|\F|}{|H|}=O\l(\l(\frac{\tau(n,m)}{m}\r)^{i-1}\frac{|\F|}{|H|}\r),
$$
where the last equal sign comes from using the second term in the maximum in $\tau(n,m)$.

Note that by \Cref{equation:supersaturation},
$$
1=\Delta_{|F|}(\F)\le C_18^{|F|}m^{-|F|+1}n^{r|F|-v(F)}\frac{|\F|}{|H|}=O\l(\l(\frac{\tau(n,m)}{m}\r)^{|F|-1}\frac{|\F|}{|H|}\r),
$$
where the last equal sign comes from using the first term in the maximum in $\tau(n,m)$. Therefore, the proof is completed by taking sufficiently large $C'$.
\end{proof}

\subsection{Expansions of Complete Bipartite Graphs}
Here we prove \Cref{thm:Kst} giving tight bounds for random Tur\'an numbers of $K_{s,t}$ expansions for $r$ sufficiently large.  We need the following lemma.
\begin{lem}\label{lemma:tighttree_Kst}
When $r\ge s+1$, $K^{(r)}_{s,t}$ is a spanning subgraph of a tight $r$-tree.
\end{lem}

\begin{proof}
By \Cref{lem:spanningTreeInduction} it suffices to construct a tight $(s+1)$-tree whose 2-shadow contains a copy of $K_{s,t}$. Let $S$ and $T$ be the two parts of $K_{s,t}$ such that $|S|=s$ and $|T|=t$, where $S=\{v_1,\dots,v_s\}$ and $T=\{w_1,\dots,w_t\}$. Pick the $(s+1)$-edges $e_i=S\cup \{w_i\}$ for $1\le i\le t$. It is easy to check that $e_1,\dots,e_t$ form a tight $(s+1)$-tree whose 2-shadow contains all $v_iw_j$, where $1\le i\le s$ and $1\le j\le t$.
\end{proof}

Next we prove balanced supersaturation results for $K^{(r)}_{s,t}$ using The balanced supersaturation results for $K_{s,t}$ proved by Morris and Saxton (\Cref{theorem:MorrisSaxtonSupersat}) together with the Shadow Expansion method (\Cref{lemma:BSviaSHADOW}) and the Greedy Expansion method (\Cref{lemma:BSviaGreedy}).

\begin{lem}\label{lemma:BS_Kst}
Let $t\ge s\ge 2$ and $r\ge s+1$. There exist sufficiently large $C_r$ such that for the functions
$$
M_r(n)=C_rn^{r-1},
$$
$$
\gamma_r(n)=(\log n)^{C_r},
$$
and
$$
\tau_r(n,m)=\max\l\{n^{\frac{2st-s-t}{st-1}},~m^{-\frac{(s-1)^2t}{2(st-1)(r-s)+(s-1)^2t}}n^{\frac{(s+1)((st-1)(r-s)+(s-1)^2t)}{2(st-1)(r-s)+(s-1)^2t}}\r\}(\log n)^{C_r},
$$
$K^{(r)}_{s,t}$ is $(M_r,\gamma_r(n),\tau_r(n,m))$-balanced. 
\end{lem}

\begin{proof}
We will first use \Cref{lemma:BSviaSHADOW} to obtain balanced supersaturation results for $K^{(s)}_{s,t}$, and then apply \Cref{lemma:BSviaGreedy} to deal with $K^{(r)}_{s,t}$ when $r\ge s+1$.

Let $M_2,C_2,\tau_2$ be as in Theorem~\ref{theorem:MorrisSaxtonSupersat} and let $\gam_2(n)$ be constantly $C_2$.  In order to apply \Cref{lemma:BSviaSHADOW} with $F=K_{s,t}$ and $r=s$, we need to check the four properties (a), (b), (c) and (d).  Property (a) of $F$ being $(M_2,\gam_2,\tau_2)$-balanced is immediate from Theorem~\ref{theorem:MorrisSaxtonSupersat}, property (b) of having $M_2(n)n^{-\frac{s-2}{s-1}}=n^{1+\frac{1}{s(s-1)}}$ being non-increasing holds, and property (c) of $\gam_2(n)$ being non-decreasing is immediate.  For property (d), we observe that $\tau_2(n,m)$ is non-increasing with $m$ and that for every $s\ge 2$
$$
\tau_2(nx,mx^{\frac{s-2}{s-1}})=\max\l\{(nx)^{\frac{2st-s-t}{st-1}},~m^{-s+1}n^{2s-1}x^{s+1}\r\}
$$
is increasing with respect to $x$. So property (d) also holds. Hence by \Cref{lemma:BSviaSHADOW}, there exists sufficiently large $C_s$ such that for the functions
$$
M_s(n)=\max\l\{M_2(n)^{s-1}n^{-s+2},~n^{s-1}\r\}(\log n)^{C_s}=n^{s-1+\frac{1}{s}}(\log n)^{C_s},
$$
$$
\gamma_s(n)=(\log n)^{C_s},
$$
and
$$
\tau_s(n,m)=\tau_2(n,n^{\frac{s-2}{s-1}}m^{\frac{1}{s-1}})(\log n)^{C_s}=\max\l\{n^{\frac{2st-s-t}{st-1}},~m^{-{1}}n^{s+1}\r\}(\log n)^{C_s},
$$
$K^{(s)}_{s,t}$ is $(M_s,\gamma_s,\tau_s)$-balanced.

Note that this balanced supersaturation result for $K^{(s)}_{s,t}$ is almost the same as Lemma~\ref{lemma:BS_Kst} when $r=s$; the only difference is that $M_s(n)$ is slightly larger than expected. This forbids us from using induction on $r$ with $r=s$ as the base case. To cope with this obstacle, we will prove an inductive claim with a slightly weaker requirement on $M_{r-1}$. Precisely, it suffices to prove the following claim.
\begin{claim}
Let $t\ge s\ge 2$ and $r\ge s+1$. If there exist $C_{r-1}$ such that for the functions
$$
M_{r-1}(n)=\max\{C_{r-1}n^{r-2},n^{s-1+\frac{1}{s}}(\log n)^{C_{r-1}}\},
$$
$$
\gamma_{r-1}(n)=(\log n)^{C_{r-1}},
$$
and
$$
\tau_{r-1}(n,m)=\max\l\{n^{\frac{2st-s-t}{st-1}},~m^{-\frac{(s-1)^2t}{2(st-1)(r-1-s)+(s-1)^2t}}n^{\frac{(s+1)((st-1)(r-1-s)+(s-1)^2t)}{2(st-1)(r-1-s)+(s-1)^2t}}\r\}(\log n)^{C_{r-1}},
$$
$K^{(r-1)}_{s,t}$ is $(M_{r-1},\gamma_{r-1}(n),\tau_{r-1}(n,m))$-balanced, then there exist $C_r$ such that for the functions
$$
M_r(n)=C_rn^{r-1},
$$
$$
\gamma_r(n)=(\log n)^{C_r},
$$
and
$$
\tau_r(n,m)=\max\l\{n^{\frac{2st-s-t}{st-1}},~m^{-\frac{(s-1)^2t}{2(st-1)(r-s)+(s-1)^2t}}n^{\frac{(s+1)((st-1)(r-s)+(s-1)^2t)}{2(st-1)(r-s)+(s-1)^2t}}\r\}(\log n)^{C_r},
$$
$K^{(r)}_{s,t}$ is $(M_r,\gamma_r(n),\tau_r(n,m))$-balanced. 
\end{claim}

\begin{proof}
It is easy to check that $\gamma_{r-1}$ and $\tau_{r-1}$ satisfy properties (b) and (c) in Theorem~\ref{lemma:BSviaGreedy}. Thus by Lemma~\ref{lemma:tighttree_Kst}, we can apply Theorem~\ref{lemma:BSviaGreedy} with $F=K^{(r-1)}_{s,t}$, and we let $C_r$ be the constant given by this theorem.  We let 
$$
A(n,m)=m^{\frac{2(st-1)}{2(st-1)(r-s)+(s-1)^2t}}n^{-\frac{(s+1)(st-1)}{2(st-1)(r-s)+(s-1)^2t}}.
$$

We need to check property (d) in Theorem~\ref{lemma:BSviaGreedy}, that is, for all sufficiently large $n$ and $m\ge M_r(n)$,
$$
mn^{1-r}\le A(n,m)\le \frac{m}{M_{r-1}(n)\cdot C_r\log n}.
$$

The first inequality is equivalent to 

$$
m^{1-\frac{2(st-1)}{2(st-1)(r-s)+(s-1)^2t}}\le n^{r-1-\frac{(s+1)(st-1)}{2(st-1)(r-s)+(s-1)^2t}}.
$$
Since $m\le n^r$, it suffices to have $1\le \frac{(2r-s-1)(st-1)}{2(st-1)(r-s)+(s-1)^2t}=\frac{2(st-1)(r-s)+(s-1)(st-1)}{2(st-1)(r-s)+(s-1)^2t}$, which is equivalent to $t\ge 1$.

We still need to check
$$
A(n,m)\le \frac{m}{M_{r-1}(n)\cdot C_r\log n},
$$
which is equivalent to
$$
M_{r-1}(n)\le \frac{m^{1-\frac{2(st-1)}{2(st-1)(r-s)+(s-1)^2t}}n^{\frac{(s+1)(st-1)}{2(st-1)(r-s)+(s-1)^2t}}}{C_r\log n}.
$$
Note that $m\ge C_rn^{r-1}$, so it suffices to check the inequality above when $m=C_rn^{r-1}$, which for some appropriate $C'$ is
$$
M_{r-1}(n)\le {C'n^{r-2+\frac{st+t+s-3}{2(st-1)(r-s)+(s-1)^2t}}}(\log n)^{-1}.
$$
Clearly, this holds when $r\ge s+2$. When $r=s+1$, we need
$$
\frac{1}{s}<\frac{st+t+s-3}{2(st-1)+(s-1)^2t},
$$
which is equivalent to
$$
0<(s-1)(s+t-2),
$$
which clearly holds.

Note that when omitting logarithmic factors, we have
\begin{equation}\label{equation:BS_Kst_2}
\begin{aligned}
\tau_{r-1}\l(n,\frac{m}{A(n,m)}\r)&=\max\l\{n^{\frac{2st-s-t}{st-1}},\l(\frac{m}{A(n,m)}\r)^{-\frac{(s-1)^2t}{2(st-1)(r-1-s)+(s-1)^2t}}n^{\frac{(s+1)((st-1)(r-1-s)+(s-1)^2t)}{2(st-1)(r-1-s)+(s-1)^2t}}\r\}\\
&=\max\l\{n^{\frac{2st-s-t}{st-1}},m^xn^y\r\},
\end{aligned}
\end{equation}
where
$$
\begin{aligned}
x&=-\frac{(s-1)^2t}{2(st-1)(r-1-s)+(s-1)^2t}\cdot\l(1-\frac{2(st-1)}{2(st-1)(r-s)+(s-1)^2t}\r)\\
&=-\frac{(s-1)^2t}{2(st-1)(r-s)+(s-1)^2t}    
\end{aligned}
$$
and
$$
\begin{aligned}
y&=-\frac{(s-1)^2t}{2(st-1)(r-1-s)+(s-1)^2t}\cdot \frac{(s+1)(st-1)}{2(st-1)(r-s)+(s-1)^2t}+\frac{(s+1)((st-1)(r-1-s)+(s-1)^2t)}{2(st-1)(r-1-s)+(s-1)^2t}\\
&=\frac{s+1}{2(st-1)(r-1-s)+(s-1)^2t}\cdot \frac{(2(st-1)(r-1-s)+(s-1)^2t)((st-1)(r-s)+(s-1)^2t)}{2(st-1)(r-s)+(s-1)^2t}\\
&=\frac{(s+1)((st-1)(r-s)+(s-1)^2t)}{2(st-1)(r-s)+(s-1)^2t)}.
\end{aligned}
$$

Also note that
\begin{equation}\label{equation:BS_Kst_1}
\begin{aligned}
A(n,m)^{-\frac{1}{d_r\l(K^{(r)}_{s,t}\r)}}&=A(n,m)^{-r+s-\frac{(s-1)^2t}{st-1}}\\
&=m^{-\frac{2(st-1)(r-s)+2(s-1)^2t}{2(st-1)(r-s)+(s-1)^2t}}n^{\frac{(s+1)((st-1)(r-s)+(s-1)^2t)}{2(st-1)(r-s)+(s-1)^2t}}\\
&=m^{x}n^y.
\end{aligned}    
\end{equation}

Recall that
$$
M_r(n)=C_rn^{r-1},
$$
$$
\gamma_r(n)=(\log n)^{C_r}\ge C_r\gamma_{r-1}(n)\log n,
$$
and
$$
\begin{aligned}
\tau_r(n,m)&=\max\l\{n^{\frac{2st-s-t}{st-1}},~m^{-\frac{(s-1)^2t}{2(st-1)(r-s)+(s-1)^2t}}n^{\frac{(s+1)((st-1)(r-s)+(s-1)^2t)}{2(st-1)(r-s)+(s-1)^2t}}\r\}(\log n)^{C_{r}}\\
&\ge \max\l\{A(n,m)^{-\frac{1}{d_r\l(K^{(r)}_{s,t}\r)}},\tau_{r-1}\l(n,\frac{m}{C_rA(n,m)\log n}\r)C_r\log n\r\}.
\end{aligned}
$$
The last inequality above comes from~(\ref{equation:BS_Kst_1}) and (\ref{equation:BS_Kst_2}). Thus by Theorem~\ref{lemma:BSviaGreedy} $K^{(r)}_{s,t}$ is $(M_r,\gamma_r,\tau_r)$-balanced.
\end{proof} 
\end{proof}

We use Lemma~\ref{lemma:BS_Kst} together with \Cref{Lemma:General Random Turan} to prove Theorem~\ref{thm:Kst}.

\begin{proof}[Proof of Theorem~\ref{thm:Kst}]
By \Cref{lem:generalLowerRandomTuran} it suffices to only prove the upper bound of Theorem~\ref{thm:Kst} for the larger range of $p$. By Lemma~\ref{lemma:BS_Kst}, we can apply Lemma~\ref{Lemma:General Random Turan} with $F=K^{(r)}_{s,t}$ which guarantees the existence of a sufficiently large constant $C_1$ such that for all $n\ge N_r$ and $m\ge C_rn^{r-1}$,
\begin{equation}\label{equation:1_Kst_greedy}
\ex(K^{(r)}_{s,t}, G^r_{n,p})\le \max\l\{{C_1}pm,n^{\frac{2st-s-t}{st-1}}(\log n)^{C_1},~m^{-\frac{(s-1)^2t}{2(st-1)(r-s)+(s-1)^2t}}n^{\frac{(s+1)((st-1)(r-s)+(s-1)^2t)}{2(st-1)(r-s)+(s-1)^2t}}(\log n)^{C_1}, \r\}.   
\end{equation}

Let $M_1$ be such that 
\begin{equation}\label{equation:2_Kst_greedy}
{C_1}pM_1=M_1^{-\frac{(s-1)^2t}{2(st-1)(r-s)+(s-1)^2t}}n^{\frac{(s+1)((st-1)(r-s)+(s-1)^2t)}{2(st-1)(r-s)+(s-1)^2t}}(\log n)^{C_1}.
\end{equation}

One can check that for some constant $C_2$
\begin{equation}\label{equation:3_Kst_greedy}
M_1\le p^{-\frac{2(st-1)(r-s)+(s-1)^2t}{2(st-1)(r-s)+2(s-1)^2t}}n^{\frac{s+1}{2}}(\log n)^{C_2}
\end{equation}

If $M_1\ge C_rn^{r-1}$, then we can apply (\ref{equation:1_Kst_greedy}), which by (\ref{equation:2_Kst_greedy}) and (\ref{equation:3_Kst_greedy}) gives 
\begin{equation}\label{equation:4_Kst_greedy}
\begin{aligned}
\ex(K^{(r)}_{s,t}, G^r_{n,p})&\le \max\l\{n^{\frac{2st-s-t}{st-1}}(\log n)^{C_1},~C_1pM_1\r\}\\
&\le \max\l\{n^{\frac{2st-s-t}{st-1}}(\log n)^{C_1},~p^{\frac{(s-1)^2t}{2(st-1)(r-s)+2(s-1)^2t}}n^{\frac{s+1}{2}}(\log n)^{C_2}\r\}.
\end{aligned}    
\end{equation}

If $M_1\le C_rn^{r-1}$, let $m=C_rn^{r-1}$. Since $m\ge M_1$, by (\ref{equation:2_Kst_greedy}) we have
$$
m^{-\frac{(s-1)^2t}{2(st-1)(r-s)+(s-1)^2t}}n^{\frac{(s+1)((st-1)(r-s)+(s-1)^2t)}{2(st-1)(r-s)+(s-1)^2t}}(\log n)^{C_1}\le{C_1}pm.
$$
Thus by (\ref{equation:1_Kst_greedy}),
\begin{equation}\label{equation:5_Kst_greedy}
\begin{aligned}
\ex(K^{(r)}_{s,t}, G^r_{n,p})&\le \max\l\{n^{\frac{2st-s-t}{st-1}}(\log n)^{C_1},~C_1pm\r\}\\
&=\max\l\{n^{\frac{2st-s-t}{st-1}}(\log n)^{C_1},~C_1C_rpn^{r-1}\r\}.
\end{aligned}    
\end{equation}

To summarize, by (\ref{equation:4_Kst_greedy}) and (\ref{equation:5_Kst_greedy}), for some sufficiently large constant $C_3$ we have 
$$
\ex(K^{(r)}_{s,t}, G^r_{n,p})\le \max\l\{n^{\frac{2st-s-t}{st-1}}(\log n)^{C_3},~p^{\frac{(s-1)^2t}{2(st-1)(r-s)+2(s-1)^2t}}n^{\frac{s+1}{2}}(\log n)^{C_3},~C_3pn^{r-1}\r\}.
$$
From ease of analysis we ignore the logarithmic terms in these expressions.  Let $\alpha=\alpha(r,s,t)>0$ be the unique value such that at $p=n^{-\alpha}$ the first and third terms equal each other and such that for smaller $p$ the first term is larger than the third, i.e.\ $\alpha=\frac{2st-s-t}{st-1}+1-r$.  Similarly there exists some $\beta=\beta(r,s,t)>0$ such that at $p=n^{-\alpha}$ the last two terms are equal and such that for all smaller $p$ the second term is larger (this exists since the exponent for $p$ and $n$ in the third term is larger for $r\ge s+1$ and $s\ge 3$).  

Observe that the second term will never achieve the maximum if $\beta\ge \alpha$, since in this case the second term will only beat the third term for $p=n^{-\gamma}$ with $\gamma\ge \beta$ but this value will be smaller than the second term at $p=n^{-\alpha}$ which is smaller than the third term which is smaller than the first term.  A laborious computation shows that $\beta\ge \alpha$ holds precisely if
\[r\ge s+1+\frac{s-1}{t-1}-\frac{s+t-2}{st-1}.\]
This inequality holds for all $r\ge s+2$ since the fractional part is at most 1 due to $s\le t$.  At $r=s+1$ it holds provided
$\frac{s-1}{t-1}-\frac{s+t-2}{st-1}\le 0$, and in fact this holds for the even larger quantity $\frac{s-1}{t-1}-\frac{s+t-2}{st-1}+\frac{1}{(t-1)(st-1)}=\frac{t(s^2-2s+3-t)}{(t-1)(st-1)}$ provided $t\ge s^2-2s+3$ as desired.
\end{proof}

\end{document}